\documentclass{amsart}

\usepackage{tikz}
\usepackage{graphicx}
\usepackage{amssymb}
\usepackage{hyperref}
\usepackage[latin1]{inputenc} 
\newtheorem{theorem}{Theorem}[section]
\newtheorem*{mydef}{Theorem A}

\newtheorem*{mydef4}{Theorem B}

\newtheorem{lemma}[theorem]{Lemma}

\theoremstyle{definition}
\newtheorem{definition}[theorem]{Definition}
\newtheorem{example}[theorem]{Example}

\theoremstyle{remark}
\newtheorem{remark}[theorem]{Remark}
\theoremstyle{proposition}
\newtheorem{proposition}{Proposition}
\theoremstyle{Corollary}
\newtheorem{corollary}{Corollary}

\numberwithin{equation}{section}

	\def\sgn{\operatorname{sgn}}
\begin{document}
	
\title{Dynamical characterization of initial segments \\ of the Markov and Lagrange Spectra}

\author{Davi Lima}

\address{Av Lourival Melo Mota s/n}
\curraddr{Instituto de Matem\'atica, Universidade Federal de Alagoas-Brazil}

\email{davimat@impa.br}
\thanks{The first author was partially supported by CAPES}


\author{Carlos Gustavo Moreira}
\address{Estrada Dona Castorina, 110. Rio de Janeiro, Rio de Janeiro-Brazil.}
\curraddr{Instituto Nacional de Matem\'atica Pura e Aplicada, IMPA.}
\email{gugu@impa.br}
\thanks{The second author was partially supported by CNPq and Faperj.}

\date{\today}




\maketitle

\begin{abstract}
We prove that, for every $k\ge 4$, the sets $M(k)$ and $L(k)$, which are Markov and Lagrange dynamical spectra related to conservative horseshoes and associated to continued fractions with coefficients bounded by $k$ coincide with the intersections of the classical Markov and Lagrange spectra with $(-\infty,\sqrt{k^2+4k}]$. We also observe that, despite the corresponding statement is also true for $k=2$, it is false for $k=3$.
\end{abstract}

\section{Introduction}
The classical Lagrange and Markov spectra are closed subsets of the real line related to Diophantine approximations. They arise naturally in the study of rational approximations of irrational numbers and of indefinite binary quadratic forms, respectively. Perron gave in \cite{P} dynamical characterizations of these classical spectra in terms of symbolic dynamical systems, which inspired a general definition of dynamical Markov and Lagrange spectra by Moreira in \cite{M1} and \cite{M2}: given a map $\psi:X \to X$ and a function $f:X\to \mathbb R$, we define the associated dynamical Markov and Lagrange spectra as $M(f,\psi)=\{\text{sup}_{n\in \mathbb N}f(\psi^n(x)), x\in X\}$ and \hfill \break
$L(f,\psi)=\{\text{limsup}_{n\to\infty}f(\psi^n(x)), x\in X\}$, respectively. 

Several results on the Markov and Lagrange dynamical spectra associated to hyperbolic dynamics 
(horseshoes) in dimension 2 analogous to previously known results on the classical spectra can be found in \cite{MR} and \cite{CMM} (in the case of this last paper in the context of {\it conservative} hyperbolic dynamics). A general result for  Markov and Lagrange dynamical spectra associated to horseshoes that are not known for the classical spectra is a phase transition property, proved by the authors in \cite{LM}, also in the context of conservative hyperbolic dynamics.  

For more informations and recent results on classical and dynamical Markov and Lagrange spectra, we refer to the book \cite{LMMR}.

In this paper we shall show that, for any $k\neq 3$, the initial segments of the classical spectra until $\sqrt{k^2+4k}$ coincide with the dynamical spectra $M(k)$ and $L(k)$ associated to a smooth real function and a horseshoe $\Lambda_k$ defined by a smooth conservative diffeomorphism, and thus these initial segments of the classical spectra are dynamical spectra associated to (conservative) hyperbolic dynamics in a natural way. These dynamical spectra $M(k)$ and $L(k)$ are also naturally associated to continued fractions with coefficients bounded by $k$. This result reinforces our view that the results in \cite{LM} provide natural conjectures for the geometrical and topological structures of the classical spectra. 

The most difficult task to prove the above results is to prove them for $k=4$. In order to do that, we need to prove Theorem A (see section 3) on the geometry of regular Cantor sets related to continued fractions with bounded coefficients.

We can state our main result as:

\begin{mydef4} There exists a $C^{\infty}$-conservative diffeomorphism $\varphi: \mathbb{S}^2\rightarrow \mathbb{S}^2$ from the sphere and a $C^{\infty}$-map $f:\mathbb{S}^2\rightarrow \mathbb{R}$ for which for any $k>1$ and $k\neq 3$ there is a horseshoe $\Lambda_k$ such that 
$$L(k)=L_f(\Lambda_k)=L\cap (-\infty, \sqrt{k^2+4k}] \quad \mbox{and} \quad M(k)=M_f(\Lambda_k)=M\cap (-\infty,\sqrt{k^2+4k}].$$
\end{mydef4}

\section{Preliminares}

In this section we remember basic definitions which are useful in the text.

\subsection{Some classical facts about continued fractions} 

The continued fraction expansion of an irrational number $\alpha$ is denoted by 
$$\alpha=[a_0;a_1,a_2,\dots] = a_0 + \frac{1}{a_1+\frac{1}{a_2+\frac{1}{\ddots}}},$$ 
so that the Gauss map $g:(0,1)\to[0,1)$, $g(x)=\dfrac{1}{x}-\left\lfloor \dfrac{1}{x}\right\rfloor$ acts on continued fraction expansions by $g([0;a_1,a_2,\dots]) = [0;a_2,\dots]$. 

Given $\alpha=[a_0;a_1,\dots, a_n, a_{n+1},\dots]$ and $\tilde{\alpha}=[a_0;a_1,\dots, a_n, b_{n+1},\dots]$ with $a_{n+1}\neq b_{n+1}$, recall that $\alpha>\tilde{\alpha}$ if and only if $(-1)^{n+1}(a_{n+1}-b_{n+1})>0$. 

For an irrational number $\alpha=\alpha_0$, the continued fraction expansion $\alpha=[a_0;a_1,\dots]$ is recursively obtained by setting $a_n=\lfloor\alpha_n\rfloor$ and $\alpha_{n+1} = \frac{1}{\alpha_n-a_n} = \frac{1}{g^n(\alpha_0)}$. The rational approximations  
$$\frac{p_n}{q_n}:=[a_0;a_1,\dots,a_n]\in\mathbb{Q}$$ 
of $\alpha$ satisfy the recurrence relations $p_n=a_n p_{n-1}+p_{n-2}$ and $q_n=a_n q_{n-1}+q_{n-2}$ (with the convention that $p_{-2}=q_{-1}=0$ and $p_{-1}=q_{-2}=1$). Moreover, $p_{n+1}q_n-p_nq_{n+1}=(-1)^n$ and $\alpha=\frac{\alpha_n p_{n-1}+p_{n-2}}{\alpha_n q_{n-1}+q_{n-2}}$. In particular, given $\alpha=[a_0;a_1,\dots, a_n, a_{n+1},\dots]$ and $\tilde{\alpha}=[a_0;a_1,\dots,a_n,b_{n+1},\dots]$, we have 
\begin{equation}\label{eq.Imp}
\alpha-\tilde{\alpha}=(-1)^n\frac{\tilde{\alpha}_{n+1}-\alpha_{n+1}}{q_n^2(\beta_n+\alpha_{n+1})(\beta_n+\tilde{\alpha}_{n+1})}
\end{equation} 
where $\beta_n:=\frac{q_{n-1}}{q_n}=[0;a_n,\dots,a_1]$. We also write $\beta_{\underline{a}}=[0;\underline{a}^t]=[0;a_n,...,a_1]$, if $\underline{a}=(a_1,...,a_n).$ Moreover, $\beta_0=\beta_{( \ )}=0.$

In general, given a finite string $\underline{a}=(a_1,\dots, a_l)\in(\mathbb{N}^*)^l$, we write $\underline{a}^{\ast}=(a_1,...,a_{n-1})$ and
$$[0;a_1,\dots,a_l] = \frac{p(a_1\dots a_l)}{q(a_1\dots a_l)}:=\dfrac{p_{\underline{a}}}{q_{\underline{a}}}.$$ 
By Euler's rule, $q(a_1\dots a_l) = q(a_1\dots a_m) q(a_{m+1}\dots a_l) + q(a_1\dots a_{m-1}) q(a_{m+2}\dots a_l)$ for $1\leq m<l$, and $q(a_1\dots a_l) = q(a_l\dots a_1)$. 

If $\theta=(a_1,a_2,...)$ we write $\theta^t=(...,a_2,a_1)$. In particular, if $\underline{b}=(b_1,...,b_n)$ and $\theta=(a_n)_n$ we have
$(\underline{b}\theta)^t=(...,a_m,...,a_1,b_n,...,b_1)$. 

\subsection{Classical and Dynamical Markov and Lagrange spectra}

\subsubsection{Perron characterization of the classical Markov and Lagrange spectra}

Given a bi-infinite sequence $\theta=(\theta_n)_{n\in\mathbb{Z}}\in(\mathbb{N}^*)^{\mathbb{Z}}$, let 
$$\lambda_i(\theta):=[a_i;a_{i+1},a_{i+2},\dots]+[0;a_{i-1}, a_{i-2},\dots].$$
The Markov value $m(\theta)$ of $\theta$ is $m(\theta):=\sup\limits_{i\in\mathbb{Z}} \lambda_i(\theta)$. The Lagrange value $\ell(\theta)$ is $\ell(\theta):=\limsup_{n\to \infty} \lambda_i(\theta)$

The Markov spectrum is the set 
$$M:=\{m(\theta)<\infty: \theta\in(\mathbb{N}^*)^{\mathbb{Z}}\}$$
 and the Lagrange spectrum is the set
 $$L:=\{\ell(\theta)<\infty: \theta\in(\mathbb{N}^*)^{\mathbb{Z}}\}.$$
 
\subsubsection{Dynamical Markov and Lagrange spectra associated to horseshoes}

 Let $\varphi:M\rightarrow M$ be a $C^2$ diffeomorphism from a surface $M$ with a compact invariant set $\Lambda$ (for example a horseshoe) and let $f:M\rightarrow \mathbb{R}$ be a continuous function. We define for $x\in M$, \textbf{the Lagrange value} of $x$ associated to $f$ and $\varphi$ as being the number $\ell_{f,\varphi}(x)=\limsup_{n\to \infty}f(\varphi^n(x))$. Similarly, the \textbf{Markov value} of $x$ associated to $f$ and $\varphi$ is the number $m_{f,\varphi}(x)=\sup_{n\in \mathbb{Z}}f(\varphi^n(x))$. The sets
$$L_f(\Lambda)=\{\ell_{f,\varphi}(x);x\in \Lambda\}$$
and
$$M_f(\Lambda)=\{m_{f,\varphi}(x);x\in \Lambda\}$$
are called \textbf{Lagrange Spectrum} of $(f,\Lambda)$ and \textbf{Markov Spectrum} of $(f, \Lambda)$, respectively.

\subsubsection{The spectra $L(k)$ and $M(k)$}

Given an integer $k\ge 2$, let $C(k)$ be the regular Cantor set formed by the numbers in $(0,1)$ whose coefficients in the continued fractions expansion are bounded by $k$.

We write $\tilde{C}(N)=\{1,2,...,N\}+C(N)$, $\Lambda_N=C(N)\times \tilde{C}(N)$ and $f(x,y)=x+y$. If $x=[0;a_1(x),a_2(x),...]$ then we take $\varphi:\Lambda_N \rightarrow \Lambda_N$ given by
	\begin{equation}\label{dif}
	\varphi(x,y)=(g(x),a_1(x)+1/y).
	\end{equation}
We note that $\varphi$ can be extended to a $C^{\infty}$-diffeomorphism on the 2-dimensional sphere $\mathbb{S}^2$. 

Notice also that $\varphi$ is conjugated to the restriction to $C(N)\times C(N)$ of the map $\psi:(0,1)\times(0,1)\to [0,1)\times(0,1)$ given by 
$$\psi(x,y)=(\{\frac1{x}\},\frac1{y+\lfloor 1/x\rfloor})$$ and following \cite{Ar} and \cite{S.ITO} we know that $\psi$ has an invariant measure equivalent to the Lebesgue measure, in particular, $\varphi$ also has an invariant measure equivalent to the Lebesgue measure and then $\varphi$ is conservative. 

Indeed, if $\Sigma=\{(x,y)\in {\mathbb R}^2|0<x<1, 0<y<1/(1+x)\}$ and $T:\Sigma\to\Sigma$ is given by
$$T(x,y)=(\{\frac1{x}\},x-x^2 y),$$
then $T$ preserves the Lebesgue measure in the plane.
If $h:\Sigma\to [0,1)\times(0,1)$ is given by $h(x,y)=(x,y/(1-xy))$ then $h$ is a conjugation between $T$ and $\psi$ (and thus $\psi$ preserves the smooth measure $h_*$(Leb)).

Moreover, $\max f(\Lambda_N)=2B_N+N=\sqrt{N^2+4N}.$ In particular, $\max f(\Lambda_4)=\sqrt{32}$.
We know that $\varphi|_{\Lambda_N}$ is topologically conjugated to $\sigma:\{1,2,...,N\}^{\mathbb{Z}}\rightarrow \{1,2,..,N\}^{\mathbb{Z}}$. Moreover, if $\tilde{f}:\{1,2,...,N\}^{\mathbb{Z}}\rightarrow \mathbb{R}$ is given by 
	$$\tilde{f}(\theta)=[a_0(\theta);a_1(\theta),a_2(\theta),...]+[0;a_{-1}(\theta),a_{-2}(\theta),...]=\lambda_0(\theta),$$ where $\theta=(a_i(\theta))_{i\in \mathbb{Z}},$ then
	$$L_f(\Lambda_N)=\{\ell_{\tilde{f},\sigma}(\theta);\theta\in \{1,2,...,N\}^{\mathbb{Z}}\}=:L(N)$$
and
	$$M_f(\Lambda_N)=\{m_{\tilde{f},\sigma}(\theta);\theta\in \{1,2,...,N\}^{\mathbb{Z}}\}=:M(N).$$

It follows that we have the following characterization of $M(k)$ and $L(k)$, analogous to Perron characterization of the classical Markov and Lagrange spectra:
$$M(k)=\{m(\theta)<\infty: \theta\in\{1,2,\dots,k\}^{\mathbb{Z}}\}$$
$$L(k)=\{\ell(\theta)<\infty: \theta\in\{1,2,\dots,k\}^{\mathbb{Z}}\}.$$

Our main result is thus reduced to prove that, for any $k>1$ and $k\neq 3$,
$$L(k)=L_f(\Lambda_k)=L\cap (-\infty, \sqrt{k^2+4k}] \quad \mbox{and} \quad M(k)=M_f(\Lambda_k)=M\cap (-\infty,\sqrt{k^2+4k}].$$

\subsection{Thickness of regular Cantor sets}

\begin{definition}\label{cantor}
	A set $K\subset \mathbb{R}$ is called a $C^s$-\textbf{regular Cantor set}, $s\geq 1$, if there exists a collection $\mathcal{P}=\{I_1,I_2,...,I_r\}$ of compacts intervals and a $C^s$ expanding map $\psi$, defined in a neighbourhood of $\displaystyle \cup_{1\leq j\leq r}I_j$ such that
	
	\begin{enumerate}
		\item[(i)] $K\subset \cup_{1\leq j\leq r}I_j$ and $\cup_{1\leq j\leq r}\partial I_j\subset K$,
		
		\item[(ii)] For every $1\leq j\leq r$ we have that $\psi(I_j)$ is the convex hull of a union of $I_t$'s, for $l$ sufficiently large $\psi^l(K\cap I_j)=K$ and $$K=\bigcap_{n\geq 0}\psi^{-n}(\cup_{1\leq j\leq r}I_j).$$
	\end{enumerate}
	
\end{definition}

\begin{example}\label{ex1}
Let $g$ be the Gauss map. We have $C(N)=\{x=[0;a_1,a_2,...]; a_i\le N, \forall i\ge 1\}$. Then, 
	$$C(N)=\bigcap_{n\ge 0}g^{-n}(I_N\cup... \cup I_1),$$
where $I_j=[a_j,b_j]$ and $a_j=[0;j,\overline{1,N}]$ and $b_j=[0;j,\overline{N,1}].$ That is, $C(N)$ is a regular Cantor set. The set $K(\underline{b})$ of the introduction is also a regular Cantor set, see its construction in the section below.


\end{example}

\begin{definition}\label{def.1}
A gap $U$, of a Cantor set $K$, is a connected component of $\mathbb{R}\setminus K$.  Given a bounded gap $U$ of $K$ we have two intervals $L_U$ and $R_U$ that separate it from the gaps bigger than it are closer.

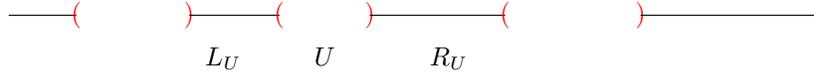
\begin{center}
\begin{tikzpicture}[domain=0:0.5,xscale=3,yscale=3]
\draw (-2.3,0) -- (-2,0);
\node[red] at (-2,0) { ( };
\node[red] at (-1.5,0) { ) };

\draw (-1.5,0) -- (-1.1,0);
\node[below] at (-1.35,-0.1) {$L_{U}$};

\node[red] at (-1.1,0) { ( };
\node[red] at (-0.7,0) { ) };
\node[below] at (-0.9,-0.1) {$U$};

\draw (-0.7,0)--(-0.1,0);
\node[below] at (-0.35,-0.1) {$R_{U}$};

\node[red] at (-0.1,0) { ( };
\node[red] at (0.5,0) { ) };
\draw (0.5,0) -- (1.3,0);

\end{tikzpicture}
	\end{center}

 The \textbf{thickness} of $K$ is 
$$\tau(K)=\min\{\tau_L(K),\tau_R(K)\},$$
where 
$$\tau_L(K)=\inf\{|L_U|/|U|; U \ \mbox{is a bounded gap}\}$$
 and 
$$\tau_R(K)=\inf\{|R_U|/|U|; U \ \mbox{is a bounded gap}\}.$$
Moreover $L_U$ and $R_U$ are the \textbf{bridges} of $U$.
\end{definition}
We use the well known version of a Newhouse's theorem
\begin{theorem}\label{NH}
Let $K$ and $\tilde{K}$ be two regular Cantor sets such that $\tau(K)\tau(\tilde{K})\ge 1$ and $I$ and $\tilde{I}$ is the convex hull of $K$ and $\tilde{K}$. If $|I|>|\tilde{O}|$ where $\tilde{O}$ is the largest bounded gap $\tilde{O}$ of $\tilde{K}$  and $|\tilde{I}|>|O|$ where $O$ is the largest bounded gap $O$ of $K$ then, $K+\tilde{K}=I+\tilde{I}$.
\end{theorem}

In the sections \ref{section3} and \ref{section4} we use extensively the equality (\ref{eq.Imp}) without weaving more words on it.
\section{Combinatorial construction of $K(\underline{b})$:\\ Theorem A}\label{section3}

We are interested in the sets 
	$$K(\underline{b}):=K(\underline{b},\mathcal{B})=\{x\in [0,1]; x=[0;\underline{b},\theta], b_s\theta\in D \}$$
	where $\underline{b}=(b_1,...,b_s)\in \{1,2,3,4\}^s$ and $D$ is the set of words $\theta=(a_n)_{n\in \mathbb{N}}$ such that $(a_i,a_{i+1})\notin \mathcal{B}$, $\mathcal{B}=\{(1,4),(2,4)\}$. We prove that for any $\underline{b}\in \{1,2,3,4\}^s$ the thickness of $K(\underline{b})$ is greater than $1$. Precisely
	
\begin{mydef}\label{T1}
We have that $\tau(K(\underline{b}))>1.03$ for any $\underline{b}\in \{1,2,3,4\}^s, \ s\in\mathbb{N}$.
\end{mydef}
The above theorem will allow us to conclude that $K(\underline{b})+K(\underline{c})$ is an interval under reasonable hypotesis and that some choice of them glue in a nice way - see Proposition \ref{p1}, and we will use this to prove the main results for $k=4$. Then we generalize this in the Theorem \ref{theorem4.8} for arbitrary values of $k$, and we use this to conclude the proof of the main results of this paper.


First we suppose $s$ even. The case $s$ odd is analogous. 

\subsection{Intervals of first type} For any $r$ even we put
\begin{equation}\label{I.1}
[[0;\underline{b},a_1,a_2,...,a_r,4,\overline{1,3}],[0;\underline{b},a_1,a_2,...,a_r,\overline{1,3}]]
\end{equation}
when $r=0$ and $b_s=3$ or $b_s=4$ or $a_r=3$ or $a_r=4$.

For $r$ odd we put
\begin{equation}\label{I.1'}
[[0;\underline{b},a_1,a_2,...,a_r,\overline{1,3},[0;\underline{b},a_1,a_2,...,a_r,4,\overline{1,3}]]
\end{equation}
when $r=0$ or $a_r=3$ or $a_r=4$.

\subsection{Intervals of second type} For any $r$ even we put
\begin{equation}\label{I.2}
[[0;\underline{b},a_1,a_2,...,a_r,\overline{3,1}],[0;\underline{b},a_1,a_2,...,a_r,\overline{1,3}]]
\end{equation}
when $r=0$ and $b_s=1$ or $b_s=2$ or $a_r=1$ or $a_2=2$.

For $r$ odd we put

\begin{equation}\label{I.2'}
[[0;\underline{b},a_1,a_2,...,a_r,\overline{1,3}],[0;\underline{b},a_1,a_2,...,a_r,\overline{3,1}]]
\end{equation}
when $a_r=1$ or $a_2=2$.

\subsection{Dividing intervals of first type} We write $\underline{a}=(a_1,a_2,...,a_r)$. First we suppose $r$ even and $a_r=3,4$. An interval of first type $I_{\underline{ba}}$ is divided omitting three open intervals
\begin{enumerate}\label{O.1}
\item[(i)] $O^1_{\underline{ba}}=([0;\underline{b},\underline{a},2,\overline{3,1}],[0;\underline{b},\underline{a},1,\overline{1,3}])$
\item[(ii)] $O^2_{\underline{ba}}=([0;\underline{b},\underline{a},3,4,\overline{1,3}],[0;\underline{b},\underline{a},2,\overline{1,3}])$
\item[(iii)] $O^3_{\underline{ba}}=([0;\underline{b},\underline{a},4,4,\overline{1,3}],[0;\underline{b},\underline{a},3,\overline{1,3}])$

\end{enumerate}

For instance, in this case the interval $I_{ba}=I_{\underline{ba}4}\cup O^3_{\underline{ba}}\cup I_{\underline{ba}3}\cup O^2_{\underline{ba}}\cup I_{\underline{ba}2}\cup O^1_{\underline{ba}}\cup I_{\underline{ba}1}$, where
 \begin{enumerate}\label{Construction.1}
\item[(a)] $I_{\underline{b}\underline{a}j}=[[0;\underline{b},\underline{a},j,\overline{1,3}],[0;\underline{b},\underline{a},j,\overline{3,1}]]$ if $j=1,2$

\item[(b)] $I_{\underline{b}\underline{a}j}=[[0;\underline{b},\underline{a},j,\overline{1,3}],[0;\underline{b},\underline{a},j,4,\overline{1,3}]]$ if $j=3,4$.
\end{enumerate}

\begin{center}
\begin{tikzpicture}[domain=0:0.5,xscale=3,yscale=3]
\draw (-2.3,0.3) -- (1.3,0.3);
\node at (-2.3,0.4) {$I_{\underline{ba}}$};
\draw (-2.3,0) -- (-1.9,0);
\node[below] at (-2.2,-0.1) {$I_{\underline{ba}4}$};
\node[red] at (-1.9,0) { ( };
\node[red] at (-1.6,0) { ) };
\node[below] at (-1.75,-0.1) {$O^{3}_{\underline{ba}}$};

\draw (-1.6,0) -- (-1.1,0);
\node[below] at (-1.35,-0.1) {$I_{\underline{ba}3}$};

\node[red] at (-1.1,0) { ( };
\node[red] at (-0.7,0) { ) };
\node[below] at (-0.9,-0.1) {$O^{2}_{\underline{ba}}$};

\draw (-0.7,0)--(-0.1,0);
\node[below] at (-0.35,-0.1) {$I_{\underline{ba}2}$};

\node[red] at (-0.1,0) { ( };
\node[red] at (0.5,0) { ) };
\node[below] at (0.2,-0.1) {$O^1_{\underline{ba}}$};
\node[below] at (0.9,-0.1) {$I_{\underline{ba}1}$};
\draw (0.5,0) -- (1.3,0);

\end{tikzpicture}
	\end{center}

Next, if $r$ is odd we divide $I_{\underline{ba}}$ as follows

\begin{enumerate}\label{O.2}
\item[(i)] $O^1_{\underline{ba}}=([0;\underline{b},\underline{a},1,\overline{1,3}],[0;\underline{b},\underline{a},2,\overline{3,1}])$
\item[(ii)] $O^2_{\underline{ba}}=([0;\underline{b},\underline{a},2,\overline{1,3}],[0;\underline{b},\underline{a},3,4,\overline{1,3}])$
\item[(iii)] $O^3_{\underline{ba}}=([0;\underline{b},\underline{a},3,\overline{1,3}],[0;\underline{b},\underline{a},4,4,\overline{1,3}])$
\end{enumerate}
and $I_{ba}=I_{\underline{ba}1}\cup O^1_{\underline{ba}}\cup I_{\underline{ba}2}\cup O^2_{\underline{ba}}\cup I_{\underline{ba}3}\cup O^3_{\underline{ba}}\cup I_{\underline{ba}4}$.
\subsection{Dividing intervals of second type} First for $r$ even. In this case an interval $I_{\underline{ba}}$ of second type is divided omitting two open intervals

\begin{enumerate}\label{O.3}
\item[(i')] $O^1_{\underline{ba}}=[[0;\underline{b},\underline{a},2,\overline{3,1}],[0;\underline{b},\underline{a},1,\overline{1,3}]]$

\item[(ii')] $O^2_{\underline{ba}}=[[0;\underline{b},\underline{a},3,4,\overline{1,3}],[0;\underline{b},\underline{a},2,\overline{1,3}]].$
\end{enumerate}

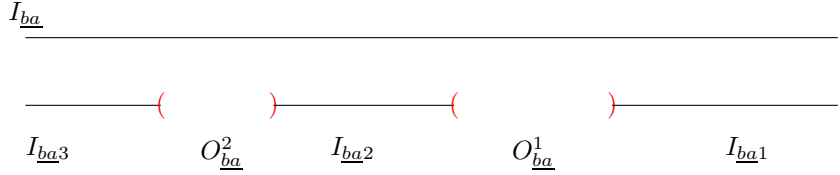
\begin{center}
\begin{tikzpicture}[domain=0:0.5,xscale=3,yscale=3]
\draw (-2.3,0.3) -- (1.3,0.3);
\node at (-2.3,0.4) {$I_{\underline{ba}}$};
\draw (-2.3,0) -- (-1.7,0);
\node[below] at (-2.2,-0.1) {$I_{\underline{ba}3}$};
\node[red] at (-1.7,0) { ( };
\node[red] at (-1.2,0) { ) };
\node[below] at (-1.43,-0.1) {$O^{2}_{\underline{ba}}$};

\draw (-1.2,0) -- (-0.4,0);

\node[below] at (-0.85,-0.1) {$I_{\underline{ba}2}$};

\node[red] at (-0.4,0) { ( };
\node[red] at (0.3,0) { ) };
\node[below] at (-0.05,-0.1) {$O^1_{\underline{ba}}$};
\node[below] at (0.9,-0.1) {$I_{\underline{ba}1}$};
\draw (0.3,0) -- (1.3,0);

\end{tikzpicture}
	\end{center}

Finally, if $r$ is odd and $a_r=1,2$ then an interval $I_{\underline{b}\underline{a}}$ of second type is divided omitting two open intervals

\begin{enumerate}\label{O.4}
\item[(i')] $O^1_{\underline{ba}}=[[0;\underline{b},\underline{a},1,\overline{1,3}],[0;\underline{b},\underline{a},2,\overline{3,1}]]$

\item[(ii')] $O^2_{\underline{ba}}=[[0;\underline{b},\underline{a},2,\overline{1,3}],[0;\underline{b},\underline{a},3,4,\overline{1,3}]].$
\end{enumerate}
\subsection{Identifying bridges}
\begin{lemma}\label{l.1}
Let $\theta, \tilde{\theta}\in \mathbb{N}^{\mathbb{N}}$. For any $j\in \mathbb{N}$ we have
	$$[0;\underline{b},j,\theta]+[0,\underline{b},j+2,\tilde{\theta}]>[0;\underline{b},j+1,\theta]+[0;\underline{b},j+1,\tilde{\theta}].$$
\end{lemma}
\begin{proof}
It is a straightforward calculation.
\end{proof}
\begin{lemma}\label{l.2} Suppose $a_r=3,4$ we have
$|O^3_{\underline{ba}}|<|O^2_{\underline{ba}}|<|O^1_{\underline{ba}}|$.
Moreover, if $a_{r-1}=3,4$ then 
	$$|O^1_{\underline{ba}}|<|O^3_{\underline{b}\underline{a}^{\ast}}|$$ 
and if $a_{r-1}=1,2$ then
	$$|O^1_{\underline{ba}}|<|O^2_{\underline{b}\underline{a}^{\ast}}|.$$
\end{lemma}

\begin{proof}
Using Lemma \ref{l.1} note that $|O^3_{\underline{ba}}|<|O^2_{\underline{ba}}|$, because
	$$|O^2_{\underline{ba}}|=|[0;\underline{b},\underline{a},2,\overline{1,3}]-[0;\underline{b},\underline{a},3,4,\overline{1,3}]|$$
and
	$$|O^3_{\underline{ba}}|=|[0;\underline{b},\underline{a},3,\overline{1,3}]-[0;\underline{b},\underline{a},4,4,\overline{1,3}]|$$
and $\sgn([0;\underline{b},\underline{a},2,\overline{1,3}]-[0;\underline{b},\underline{a},3,4,\overline{1,3}])=\sgn([0;\underline{b},\underline{a},3,\overline{1,3}]-[0;\underline{b},\underline{a},4,4,\overline{1,3}])$

On the other hand, 
	\begin{equation}\label{eq.1}|O^1_{\underline{ba}}|=\dfrac{[2;\overline{3,1}]-[1;\overline{1,3}]}{q^2_{\underline{ba}}([2;\overline{3,1}]+\beta_{\underline{ba}})([1;\overline{1,3}]+\beta_{\underline{ba}})}=\dfrac{1+[0;\overline{3,1}]-[0;\overline{1,3}]}{q^2_{\underline{ba}}([2;\overline{3,1}]+\beta_{\underline{ba}})([1;\overline{1,3}]+\beta_{\underline{ba}})},
	\end{equation}
while
	\begin{equation}\label{eq.2}|O^2_{\underline{ba}}|=\dfrac{[3;4,\overline{1,3}]-[2;\overline{1,3}]}{q^2_{\underline{ab}}([2;\overline{1,3}]+\beta_{\underline{ba}})([3;4,\overline{1,3}]+\beta_{\underline{ba}})}=\dfrac{1+[0;4,\overline{1,3}]-[0;\overline{1,3}]}{q^2_{\underline{ba}}([2;\overline{1,3}]+\beta_{\underline{ba}})([3;4,\overline{1,3}]+\beta_{\underline{ba}})}
	\end{equation}
and immediately we have $|O^1_{\underline{ba}}|>|O^2_{\underline{ba}}|$ because 
	$$1+[0;\overline{3,1}]-[0;\overline{1,3}]>1+[0;4,\overline{1,3}]-[0;\overline{1,3}]$$ 
and 
	$$([2;\overline{3,1}]+\beta_{\underline{ba}})([1;\overline{1,3}]+\beta_{\underline{ba}})<([2;\overline{1,3}]+\beta_{\underline{ba}})([3;4,\overline{1,3}]+\beta_{\underline{ba}}).$$
	Finally, 
	$$|O^3_{\underline{b}\underline{a}^{\ast}}|=|[0;\underline{b},\underline{a}^{\ast},4,4,\overline{1,3}]-[0;\underline{b},\underline{a}^{\ast},3,\overline{1,3}]|=\dfrac{[4;4,\overline{1,3}]-[3;\overline{1,3}]}{q^2_{\underline{b}\underline{a}^{\ast}}([4;4,\overline{1,3}]+\beta_{\underline{b}\underline{a}^{\ast}})([3;\overline{1,3}]+\beta_{\underline{b}\underline{a}^{\ast}})}.$$
	Then,
	$$\dfrac{|O^3_{\underline{b}\underline{a}^{\ast}}|}{|O^1_{\underline{ba}}|}=\dfrac{q^2_{\underline{ba}}}{q^2_{\underline{b}\underline{a}^{\ast}}}\cdot \dfrac{1+[0;4,\overline{1,3}]-[0;\overline{1,3}]}{1+[0;\overline{3,1}]-[0;\overline{1,3}]}\cdot \dfrac{([2;\overline{3,1}]+\beta_{\underline{ba}})([1;\overline{1,3}]+\beta_{\underline{ba}})}{([4;4,\overline{1,3}]+\beta_{\underline{b}\underline{a}^{\ast}})([3;\overline{1,3}]+\beta_{\underline{b}\underline{a}^{\ast}})}.$$
Note that $0.2\le \beta_{\underline{c}}\le 0.25$, for any $\underline{c}\in \cup_{m\ge 0}\{1,2,3,4\}^m$ and then 
	$$\dfrac{([2;\overline{3,1}]+\beta_{\underline{ba}})([1;\overline{1,3}]+\beta_{\underline{ba}})}{([4;4,\overline{1,3}]+\beta_{\underline{b}\underline{a}^{\ast}})([3;\overline{1,3}]+\beta_{\underline{b}\underline{a}^{\ast}})}>0.33.$$
Moreover, $q_{\underline{ba}}=a_rq_{\underline{b}\underline{a}^{\ast}}+q_{\underline{b}\underline{a}^{\ast \ast}}> 3q_{\underline{b}\underline{a}^{\ast}}$, because $a_r=3,4$. This implies that
	$$\dfrac{|O^3_{\underline{b}\underline{a}^{\ast}}|}{|O^1_{\underline{ba}}|}>9\cdot 0.88 \cdot 0.33>1.$$
If $a_{r-1}=1,2$ then 
	$$|O^2_{\underline{b}\underline{a}^{\ast}}|=\dfrac{1+[0;4,\overline{1,3}]-[0;\overline{1,3}]}{q^2_{\underline{b}\underline{a}^{\ast}}([3;4,\overline{1,3}]+\beta_{\underline{b}\underline{a}^{\ast}})([2;\overline{1,3}]+\beta_{\underline{b}\underline{a}^{\ast}})}.$$
Note that
	$$\dfrac{|O^2_{\underline{b}\underline{a}^{\ast}}|}{|O^1_{\underline{b}\underline{a}}|}=\dfrac{q^2_{\underline{b}\underline{a}}}{q^2_{\underline{b}\underline{a}^{\ast}}}\cdot\dfrac{1+[0;4,\overline{1,3}]-[0;\overline{1,3}]}{1+[0;\overline{3,1}]-[0;\overline{1,3}]}\cdot \dfrac{([2;\overline{3,1}]+\beta_{\underline{ba}})([1;\overline{1,3}]+\beta_{\underline{ba}})}{([3;4,\overline{1,3}]+\beta_{\underline{b}\underline{a}^{\ast}})([2;\overline{1,3}]+\beta_{\underline{b}\underline{a}^{\ast}})}.$$
Since $a_r=3,4$ and 
	$$\dfrac{([2;\overline{3,1}]+\beta_{\underline{ba}})([1;\overline{1,3}]+\beta_{\underline{ba}})}{([3;4,\overline{1,3}]+\beta_{\underline{b}\underline{a}^{\ast}})([2;\overline{1,3}]+\beta_{\underline{b}\underline{a}^{\ast}})}>\dfrac{([2;\overline{3,1}]+\beta_{\underline{ba}})([1;\overline{1,3}]+\beta_{\underline{ba}})}{([4;4,\overline{1,3}]+\beta_{\underline{b}\underline{a}^{\ast}})([3;\overline{1,3}]+\beta_{\underline{b}\underline{a}^{\ast}})}>0.33$$
	we have $|O^2_{\underline{b}\underline{a}^{\ast}}|>|O^1_{\underline{b}\underline{a}}|.$

\end{proof}

\begin{corollary}\label{c.1} Let $L_U$ and $R_U$ be the bridges associated to the bounded gap $U$ of $K$. Then for $a_r=3,4$ we have

\begin{enumerate}

	\item[(i)] $L^1_{\underline{ba}}:=L_{O^1_{\underline{ba}}}=I_{\underline{ba4}}\cup O^3_{\underline{ba}}\cup I_{\underline{ba3}}\cup O^2_{\underline{ba}}\cup I_{\underline{ba2}}$ and $R^1_{\underline{ba}}:=R_{O^1_{\underline{ba}}}=I_{\underline{ba1}}.$


	\item[(ii)] $L^2_{\underline{ba}}:=L_{O^2_{\underline{ba}}}=I_{\underline{ba4}}\cup O^3_{\underline{ba}}\cup I_{\underline{ba3}}$ and $R^2_{\underline{ba}}:=R_{O^2_{\underline{ba}}}=I_{\underline{ba2}}$
	
	\item[(iii)] $L^3_{\underline{ba}}:=L_{O^3_{\underline{ba}}}=I_{\underline{ba}4}$ and $R^3_{\underline{ba}}:=R_{O^3_{\underline{ba}}}=I_{\underline{ba}3}$

\end{enumerate}

\end{corollary}

\begin{lemma}\label{l.3}
Suppose $a_r=1,2$. We have $|O^2_{\underline{ba}}|<|O^1_{\underline{ba}}|.$ Moreover, if $a_{r-1}=3,4$ then
	$$|O^1_{\underline{ba}}|<|O^3_{\underline{b}\underline{a}^{\ast}}|$$
and if $a_{r-1}=1,2$ then
	$$|O^1_{\underline{ba}}|<|O^2_{\underline{b}\underline{a}^{\ast}}|.$$
\end{lemma}

\begin{proof}
Note that $|O^2_{\underline{ba}}|=|[0;\underline{b},\underline{a},3,4,\overline{}1,3]-[0;\underline{b},\underline{a},2,\overline{1,3}]|$ and 
	$$|O^1_{\underline{ba}}|=|[0;\underline{b},\underline{a},2,\overline{3,1}]-[0;\underline{b},\underline{a},1,\overline{1,3}]|.$$
Then,
	$$|O^2_{\underline{ba}}|=\dfrac{1+[0;4,\overline{1,3}]-[0;\overline{1,3}]}{q^2_{\underline{ba}}([3;4,\overline{1,3}]+\beta_{\underline{ba}})([2;\overline{1,3}]+\beta_{\underline{ba}})}$$
while
	$$|O^1_{\underline{ba}}|=\dfrac{1+[0;\overline{3,1}]-[0;\overline{1,3}]}{q^2_{\underline{ba}}([2;\overline{3,1}]+\beta_{\underline{ba}})([1;\overline{1,3}]+\beta_{\underline{ba}})}$$
	immediately we have that
	$$|O^2_{\underline{ba}}|<|O^1_{\underline{ba}}|.$$
If $a_{r-1}=3,4$ then $|O^3_{\underline{b}\underline{a}^{\ast}}|=|[0;\underline{b},\underline{a}^{\ast},3,\overline{1,3}]-[0;\underline{b},\underline{a}^{\ast},4,4,\overline{1,3}]|.$ Therefore,
	$$|O^3_{\underline{b}\underline{a}^{\ast}}|=\dfrac{1+[0;4,\overline{1,3}]-[0;\overline{1,3}]}{q^2_{\underline{b}\underline{a}^{\ast}}([4;4,\overline{1,3}]+\beta_{\underline{b}\underline{a}^{\ast}})([3;\overline{1,3}]+\beta_{\underline{b}\underline{a}^{\ast}})}.$$
	In the same way as Lemma \ref{l.2} we have that $|O^1_{\underline{ba}}|<|O^3_{\underline{b}\underline{a}^{\ast}}|$. The prove that $|O^1_{\underline{b}\underline{a}}|<|O^2_{\underline{b}\underline{a}^{\ast}}|$ when $a_{r-1}=1,2$ is analogous. 
	
\end{proof}

\begin{corollary}\label{c.2}
As above, if $a_r=1,2$ then

\begin{enumerate}

	\item[(i)] $L^1_{\underline{ba}}:=L_{O^1_{\underline{ba}}}=I_{\underline{ba}3}\cup O^2_{\underline{ba}}\cup I_{\underline{ba}2}$ and $R^1_{\underline{ba}}:=R_{O^1_{\underline{ba}}}=I_{\underline{ba}1}.$
	
	\item[(ii)] $L^2_{\underline{ba}}:=L_{O^2_{\underline{ba}}}=I_{\underline{ba}3}$ and $R^2_{\underline{ba}}:=R_{O^2_{\underline{ba}}}=I_{\underline{ba}2}$
\end{enumerate}	

\end{corollary}

\begin{mydef}
We have that $\tau(K(\underline{b}))>1,03$ for any $s\in \mathbb{N}$.
\end{mydef}
\begin{proof}
Fix $s=|b|$ even and $a_{r}=3,4$ . We begin with $U=O^1_{\underline{ba}}$. By corollary \ref{c.1} we have $|L_U|>|R_U|$ then we only calculate $|R_U|/|U|$. Since
	$$|I_{\underline{ba}1}|=|[0;\underline{b},\underline{a},1,\overline{1,3}]-[0;\underline{b},\underline{a},1,\overline{3,1}]|$$
we have 
	$$|I_{\underline{ba}1}|=\dfrac{[1;\overline{1,3}]-[1;\overline{3,1}]}{q^2_{\underline{ba}}([1;\overline{1,3}]+\beta_{\underline{ba}})([1;\overline{3,1}]+\beta_{\underline{ba}})}.$$
Then, using \ref{eq.1} we have that 
	$$\dfrac{\vert I_{\underline{ba}1}\vert}{\vert O^1_{\underline{ba}}\vert}=\dfrac{[1;\overline{1,3}]-[1;\overline{3,1}]}{1+[0;\overline{3,1}]-[0;\overline{1,3}]}\cdot \dfrac{([2;\overline{3,1}]+\beta_{\underline{ba}})([1;\overline{1,3}]+\beta_{\underline{ba}})}{([1;\overline{1,3}]+\beta_{\underline{ba}})([1;\overline{3,1}]+\beta_{\underline{ba}})}>1.1.$$
Next, consider $U=O^2_{\underline{ba}}$. Analogously, $|L_U|>|R_U|$ and we just calculate $|R_U|/|U|$.
Note that $$|I_{\underline{ba}2}|=\dfrac{[2;\overline{1,3}]-[2;\overline{3,1}]}{q^2_{\underline{ba}}([2;\overline{1,3}]+\beta_{\underline{ba}})([2;\overline{3,1}]+\beta_{\underline{ba}})}.$$
Then, using \ref{eq.2} we have 
	$$\dfrac{\vert I_{\underline{ba}2}\vert}{\vert O^2_{\underline{ba}}\vert}=\dfrac{[2;\overline{1,3}]-[2;\overline{3,1}]}{1+[0;4,\overline{1,3}]-[0;\overline{1,3}]}\cdot \dfrac{([3;4,\overline{1,3}]+\beta_{\underline{ba}})([2;\overline{1,3}]+\beta_{\underline{ba}})}{([2;\overline{1,3}]+\beta_{\underline{ba}})([2;\overline{3,1}]+\beta_{\underline{ba}})}>1.25.$$
Finally, if $U=O^3_{\underline{ba}}$ since $\vert I_{\underline{ba}3}\vert > \vert I_{\underline{ba}4}\vert$ we need just to calculate $\vert L_U \vert / \vert U \vert$. In this case
	$$\vert I_{\underline{ba}4}\vert =\dfrac{[4;\overline{1,3}]-[4;4,\overline{1,3}]}{q^2_{\underline{ba}}([4;\overline{1,3}]+\beta_{\underline{ba}})([4;4,\overline{1,3}]+\beta_{\underline{ba}})}$$
and then
	$$\dfrac{\vert I_{\underline{ba}4}\vert}{\vert O^3_{\underline{ba}}\vert}=\dfrac{[0;\overline{1,3}]-[0;4,\overline{1,3}]}{1+[0;4,\overline{1,3}]-[0;\overline{1,3}]}\cdot \dfrac{([4;4,\overline{1,3}]+\beta_{\underline{ba}})([3;\overline{1,3}]+\beta_{\underline{ba}})}{([4;\overline{1,3}]+\beta_{\underline{ba}})([4;4,\overline{1,3}]+\beta_{\underline{ba}})}>1.39\cdot 0.78>1.08.$$
	Next, suppose $a_r=1,2$. In the same way as above, we have $\dfrac{\vert I_{\underline{ba}1}\vert}{\vert O^1_{\underline{ba}}\vert}>1.1$. On the other hand
	the bridge of $O^2_{\underline{ba}}$ on the left side is $I_{\underline{ba}3}$ and then we have
	$$\vert I_{\underline{ba}3}\vert=\vert [0;\underline{b},\underline{a},3,4,\overline{1,3}]-[0;\underline{b},\underline{a},3,\overline{1,3}] \vert=\dfrac{[3;4,\overline{1,3}]-[3;\overline{1,3}]}{q^2_{\underline{ba}}([3;4,\overline{1,3}]+\beta_{\underline{ba}})([3;\overline{1,3}]+\beta_{\underline{ba}})}.$$
	
	$$\dfrac{\vert I_{\underline{ba}3}\vert}{\vert O^2_{\underline{ba}}\vert}=\dfrac{[0;\overline{1,3}]-[0;4,\overline{1,3}]}{1+[0;4,\overline{1,3}]-[0;\overline{1,3}]}\cdot \dfrac{([3;4,\overline{1,3}]+\beta_{\underline{ba}})([2;\overline{1,3}]+\beta_{\underline{ba}})}{([3;4,\overline{1,3}]+\beta_{\underline{ba}})([3;\overline{1,3}]+\beta_{\underline{ba}})}.$$
	This implies that
	$$\dfrac{\vert I_{\underline{ba}3}\vert}{\vert O^2_{\underline{ba}}\vert}>1.395 \cdot 0.74>1.03.$$
	In particular, $\tau(K(\underline{b}))>1.03.$ The case $s=|\underline{b}|$ odd is analogous. 
\end{proof}

\section{Sums of Cantor sets \\ Theorem B}\label{section4}
We begin this section using Theorem \ref{NH} to obtain the following corollary 
\begin{corollary}\label{c}
If $K(\underline{b})$ and $K(\underline{c})$ are such that the convex hull of $K(\underline{b})$ has size greater the largest bounded gap of $K(\underline{c})$ and if the convex hull of $K(\underline{c})$ has size greater the largest bounded gap of $K(\underline{b})$ then 
$K(\underline{b})+K(\underline{c})$ is the interval 
$$[m(\underline{b})+m(\underline{c}),M(\underline{b})+M(\underline{c})],$$
where $m(\underline{d})=\min K(\underline{d})$ is the minimum of $K(\underline{b})$ and $M(\underline{d})=\max K(\underline{d})$ is the maximum of $K(\underline{d})$ and $\underline{d}=\underline{b}, \underline{c}$. 
\end{corollary}
By simplicity, we write below $\hat{K}$ the convex hull of a cantor set $K$. We write $\underline{b}^s=(1,4)_s$ and $\underline{c}^s=(1,4)_{s-1},1,3.$
\begin{lemma}\label{l.11}
Consider $X=(\cup_{j=2}^4 K(1,4,j))+K(1,3)$. Then $X$ is the interval $[[0;1,4,4,\overline{1,3}]+[0;1,3,4,\overline{1,3}], [0,1,4,2,\overline{3,1}]+[0;\overline{1,3}]].$ More generally, for $s\ge 1$ we have 
	$$X^{(s)}=(\cup_{j=2}^4 K(\underline{b}^s,j))+K(\underline{c}^s)$$
is the interval
	$$[[0;\underline{b}^s,4,\overline{1,3}]+[0;\underline{c}^s,4,\overline{1,3}],[0;\underline{b}^s,2,\overline{3,1}]+[0;\underline{c}^s,\overline{1,3}]].$$
\end{lemma}
\begin{proof} The convex hull of $ \cup_{j=2}^4 K(\underline{b}^s,j)$ is the interval $J=[[0;\underline{b},4,\overline{1,3}],[0;\underline{b},2,\overline{3,1}]]$. We have
	$$|J|=\dfrac{[4;\overline{1,3}]-[2;\overline{3,1}]}{q^2_{\underline{b}^s}([4;\overline{1,3}]+\beta_{\underline{b}})([2;\overline{3,1}]+\beta_{\underline{b}})}.$$
	Moreover, the largest gap of $K(\underline{c}^s)$ is $O^1_{\underline{c}^s}$ which has lenght
		$$|O^1_{\underline{c}^s}|=\dfrac{[2;\overline{3,1}]-[1;\overline{1,3}]}{q^2_{\underline{c}^s}([2;\overline{3,1}]+\beta_{\underline{c}^s})([1;\overline{1,3}]+\beta_{\underline{c}^s})}.$$
		In particular, 
		$$\dfrac{|J|}{|O^1_{\underline{c}^s}|}=\dfrac{q^2_{\underline{c}^s}}{q^2_{\underline{b}^s}}\cdot X\cdot Y$$
		where 
		$$X=\dfrac{[4;\overline{1,3}]-[2;\overline{3,1}]}{[2;\overline{3,1}]-[1;\overline{1,3}]}>5.349 \ \ \mbox{and} \ \ Y=\dfrac{([2;\overline{3,1}]+\beta_{\underline{c}^s})([1;\overline{1,3}]+\beta_{\underline{c}^s})}{([4;\overline{1,3}]+\beta_{\underline{b}})([2;\overline{3,1}]+\beta_{\underline{b}})}>0.41$$
		where we used that $\beta_{\underline{c}^s}=[0;3,1,(4,1)_{s-1}]>[0;3,1]$ and $\beta_{\underline{b}^s}=[0;(4,1)_s]<[0;\overline{4,1}].$ Since $\dfrac{q_{\underline{c}^s}}{q_{\underline{b}^s}}>0.79$ ( see inequality (\ref{inequ.1}) )  we have that
		$$\dfrac{|J|}{|O^1_{\underline{c}^s}|}>0.62\cdot 5.349\cdot 0.41>1.38.$$
		In a similar way we can show that the convex hull of $K(\underline{c}^s)$ is greater than the convex hull of $\cup_{j=2}^4 K(\underline{b}^s,j)$.
		By Theorem A, $\min\{|I_{\underline{b}^s2}|,|I_{\underline{b}^s3}|\}>|O^2_{\underline{b}^s}|$, $\min\{|I_{\underline{b}^s4}|,|I_{\underline{b}^s3}|\}>|O^3_{\underline{b}^s}|$ and $K(\underline{b}^s,2)$, $K(\underline{b}^s,3)$ and $K(\underline{b}^s,4)$ has thickness greater than $1$ and then we have $\cup_{j=2}^4 K(1,4,j))$ has thickness greater than $1$. Therefore, by Corollary \ref{c} $\cup_{j=2}^4 K(\underline{b}^s,j)+K(\underline{c}^s)$ is an interval.

\end{proof}
\begin{lemma}\label{l.12}
Consider $Y_{\ell,m}=K(1,4,1,\ell)+ K(1,3,m)$, for $\ell=1,2$ and $m=2,3,4$. Then $Y_{\ell,m}$ is the interval
$$[[0;1,4,1,\ell,\overline{3,1}]+[0;1,3,m,\overline{1,3}],[0;1,4,1,\ell,\overline{1,3}]+[0;1,3,m,4,\overline{1,3}]].$$
More generally, for any $s\ge 1$, $Y^{(s)}_{\ell,m}=K(\underline{b}^s,1,\ell)+K(\underline{c}^s,m)$ with $\ell=1,2$ and $m=2,3,4$ is the interval
$$[m(\underline{b}^s,1,\ell)+m(\underline{c}^s,m), M(\underline{b}^s,1,\ell)+M(\underline{c}^s,m)].$$

\end{lemma}
\begin{proof}
The convex hull of $K(\underline{c}^s,m)$ is greater than the convex hull of $K(\underline{b}^s,1,\ell)$. Moreover 
$$|O^1_{\underline{c}^s,4}|<|O^1_{\underline{c}^s,3}|<|O^1_{\underline{c}^s,2}|<|\hat{K}(\underline{b}^s,1,2)|<|\hat{K}(\underline{b}^s,1,1)|.$$ 
We just prove that $|O^1_{\underline{c}^s,2}|<|\hat{K}(\underline{b}^s,1,2)|$, the other inequalities follow the same lines.
In fact, using that $\beta_{\underline{b}^s}=[0;(4,1)_s]$ and $\beta_{\underline{c}^s}=[0;3,1,(4,1)_{s-1}]$ we have
	$$|\hat{K}(\underline{b}^s,1,2)|=\dfrac{[1;2,\overline{3,1}]-[1;2,\overline{1,3}]}{q^2_{\underline{b}^s}([1;2,\overline{3,1}]+[0;(4,1)_s])([1;2,\overline{1,3}]+[0;(4,1)_s])},$$
while
	$$|O^1_{\underline{c}^s,2}|=\dfrac{[2;1,\overline{1,3}]-[2;2,\overline{3,1}]}{q^2_{\underline{c}^s}([2;1,\overline{1,3}]+[0;3,1,(4,1)_s])([2;2,\overline{3,1}]+[0;3,1,(4,1)_s])}.$$
Note that, since $q_{\underline{b}^s}=4q_{\underline{b}^{s-1},1}+q_{\underline{b}^{s-1}}$ and $q_{\underline{c}^s}=3q_{\underline{b}^{s-1},1}+q_{\underline{b}^{s-1}}$ and then using that $\dfrac{q_{\underline{b}^{s-1}}}{q_{\underline{b}^{s-1},1}}=[0;(1,4)_{s-1},1]$ one has
	\begin{equation}\label{inequ.1}
	\dfrac{q_{\underline{c}^s}}{q_{\underline{b}^s}}=\dfrac{1}{1+\dfrac{1}{3+[0;(1,4)_{s-1},1]}}>\dfrac{1}{1+\dfrac{1}{3+[0;\overline{1,4}]}}>0.79.
	\end{equation}
	On the other hand, 
	$$\dfrac{[1;2,\overline{3,1}]-[1;2,\overline{1,3}]}{[2;1,\overline{1,3}]-[2;2,\overline{3,1}]}>0.716$$
	and
	$$\dfrac{([2;1,\overline{1,3}]+[0;3,1,(4,1)_s])([2;2,\overline{3,1}]+[0;3,1,(4,1)_s])}{[1;2,\overline{3,1}]+[0;(4,1)_s])([1;2,\overline{1,3}]+[0;(4,1)_s])}>2.75$$
	where we used that $[0;3,1,(4,1)_s]>[0;3,1,4,1]$ and $[0;(4,1)_s]<[0;\overline{4,1}].$ Therefore,
	$$\dfrac{|\hat{K}(\underline{b}^s,1,2)|}{|O^1_{\underline{c}^s,2}|}>2.75\cdot 0.71\cdot 0.62>1.2.$$
	It follows from the Corollary \ref{c} we have that $Y^{(s)}_{\ell,m}=K(\underline{b}^s,1,\ell)+K((1,4)_{s-1},1,m)$ is an interval.

\end{proof}

\begin{lemma}\label{l36}
For any $m=2,3,4$ and $s\ge 1$ we have that
	$$M(\underline{b}^s,1,1)+M(\underline{c}^s,m)>m(\underline{b}^s,1,2)+m(\underline{c}^s,m).$$
In particular,  $ \cup_{p=1}^2 K(\underline{b}^s,1,p)+K(\underline{c}^s,m)$ is the interval
	$$[[0;\underline{b}^s,1,1,\overline{3,1}]+[0;\underline{c}^s,m,\overline{1,3}],[0;\underline{b}^s,1,2,\overline{1,3}]+[0;\underline{c}^s,m,\theta^m]],$$
	where $\theta^m=4,\overline{1,3},$ if $m=3,4$ and $\theta^m=\overline{3,1},$ if $m=2$.
	
\end{lemma}
\begin{proof}
Firstly we take $m=3,4$. We shall prove that 
	$$M(\underline{c}^s,m)-m(\underline{c}^s,m)>M(\underline{b}^s,1,1)-m(\underline{b}^s,1,2).$$
	In order to do this note that
	$$A:=M(\underline{c}^s,m)-m(\underline{c}^s,m)=\dfrac{[m;\overline{1,3}]-[m;4,\overline{1,3}]}{q^2_{\underline{c}^s}([m;\overline{1,3}]+\beta_{\underline{c}^s})([m;4,\overline{1,3}]+\beta_{\underline{c}^s})}$$
	while
	$$B:=M(\underline{b}^s,1,1)-m(\underline{b}^s,1,2)=\dfrac{[1;1,\overline{1,3}]-[1;2,\overline{3,1}]}{q^2_{\underline{b}^s}([1;1,\overline{1,3}]+\beta_{\underline{b}^s})([1;2,\overline{3,1}]+\beta_{\underline{b}^s})}.$$
	This implies
	$$\dfrac{A}{B}=\dfrac{q^2_{\underline{b}^s}}{q^2_{\underline{c}^s}}\cdot X\cdot Y$$
	where
	$$X=\dfrac{[0;\overline{1,3}]-[0;4,\overline{1,3}]}{[0;1,\overline{1,3}]-[0;2,\overline{3,1}]}=5$$
	and
	$$Y=\dfrac{([1;1,\overline{1,3}]+\beta_{\underline{b}^s})([1;2,\overline{3,1}]+\beta_{\underline{b}^s})}{([m;\overline{1,3}]+\beta_{\underline{c}^s})([m;4,\overline{1,3}]+\beta_{\underline{c}^s})}.$$
	Note that $Y$ is minimum whenever $m=4$. In this case, for $s=1$ we have $Y>1.284$, while for $s>1$ using that $\beta_{\underline{b}^s}\ge [0;4,1,4,1]$ and $\beta_{\underline{c}^s}<[0;3,1,\overline{4,1}]$ we have $Y>1.27.$ If $s=1$ we have $\dfrac{q^2_{\underline{b}^s}}{q^2_{\underline{c}^s}}=1.25$, if $s>1$ we have $\dfrac{q^2_{\underline{b}^s}}{q^2_{\underline{c}^s}}>1.26$. In any case have $A/B>1.003.$
	In the case $m=2$ we have
	$$A:=M(\underline{c}^s,m)-m(\underline{c}^s,m)=\dfrac{[2;\overline{1,3}]-[2;\overline{3,1}]}{q^2_{\underline{c}^s}([2;\overline{1,3}]+\beta_{\underline{c}^s})([2;\overline{3,1}]+\beta_{\underline{c}^s})}$$
	while
	$$B:=M(\underline{b}^s,1,1)-m(\underline{b}^s,1,2)=\dfrac{[1;1,\overline{1,3}]-[1;2,\overline{3,1}]}{q^2_{\underline{b}^s}([1;1,\overline{1,3}]+\beta_{\underline{b}^s})([1;2,\overline{3,1}]+\beta_{\underline{b}^s})}.$$
	There is no difficult to see that we also have $A/B>1$.

\end{proof}

\begin{lemma}\label{l.13}
Consider $Z= (K(1,4,1,1)\cup  K(1,4,1,2)) + (K(1,3,1,2)\cup K(1,3,1,3))$. Then $Z$ is the interval 
$$[[0;1,4,1,1,4,\overline{1,3}]+[0;1,3,1,2,4,\overline{1,3}],[0;1,4,1,2,\overline{1,3}]+[0;1,3,1,3,\overline{1,3}]].$$
More generally, for $s\ge 1$, $Z^{(s)}=\cup_{p=1}^2 K(\underline{b}^s,1,p)+\cup_{q=2}^3 K((1,4)_{s-1},1,3,1,q)$ is the interval
$$[[0;(1,4)_{s},1,1,\overline{3,1}]+[0;(1,4)_{s-1},1,2,\overline{3,1}],[0;\underline{b}^s,1,2,\overline{1,3}]+[0;(1,4)_{s-1},1,3,\overline{1,3}]].$$
\end{lemma}
\begin{proof}
Let $U$ be the gap between $K((1,4)_{s-1},1,3,1,2)$ and $K((1,4)_{s-1},1,3,1,3)$, which is $O^2_{(1,4)_{s-1},1,3,1}.$

 First we prove that $|\hat{K}(\underline{b}^s,1,2)|>|U|$. In a similar way as above we have
 	$$\dfrac{|\hat{K}(\underline{b}^s,1,2)|}{|U|}=\dfrac{q^2_{\underline{c}^s}}{q^2_{\underline{b}^s}}\cdot X\cdot Y,$$
where
	$$X=\dfrac{[1;2,\overline{3,1}]-[1;2,\overline{1,3}]}{[1;2,\overline{1,3}]-[1;3,4,\overline{1,3}]}>1.79$$
and
	$$Y=\dfrac{([1;2,\overline{1,3}]+[0;3,1,(4,1)_{s-1}])([1;3,4,\overline{1,3}]+[0;3,1,(4,1)_{s-1}])}{([1;2,\overline{3,1}]+[0;(4,1)_s])([1;2,\overline{1,3}]+[0;(4,1)_s])}>0.97.$$	
	Therefore, using the inequality \ref{inequ.1}	
	$$\dfrac{|\hat{K}(\underline{b}^s,1,2)|}{|U|}>0.62\cdot 1.79\cdot 0.97>1.08.$$
	Remember that $|\hat{K}(\underline{b}^s,1,1)|>|\hat{K}(\underline{b}^s,1,2)|$.
	Next, consider $U'$ the gap between $K(\underline{b}^s,1,1)$ and $K(\underline{b}^s,1,2)$, which is $O^1_{\underline{b}^s,1}.$ In a similar way as the previous proof we can show that 
	$$|\hat{K}(\underline{c}^s,1,2)|>|U'|.$$
	This implies that our Cantor sets satisfy the hypothesis of Corollary \ref{c}, and then $Z^{(s)}$ is an interval for $s \ge 1$.	
\end{proof}

\begin{lemma}\label{l.13}
Consider $W=(K(1,4,1,1)\cup K(1,4,1,2))+K(1,3,1,1).$ Then $W$ is the interval
	$$[[0;1,4,1,1,\overline{3,1}]+[0;1,3,1,1,\overline{3,1}],[0;1,4,1,2,\overline{1,3}]+[0;1,3,1,1,\overline{1,3}]].$$
Moreover, for $s\ge 1$, $W^{(s)}=(K(\underline{b}^s,1,1)\cup K(\underline{b}^s,1,2))+K((1,4)_{s-1},1,3,1,1)$ is the interval
	$$[[0;\underline{b}^s,1,1,\overline{3,1}]+[0;(1,4)_{s-1},1,3,1,1,\overline{3,1}],[0;\underline{b}^s,1,2,\overline{1,3}]+[0;(1,4)_{s-1},1,3,1,1,\overline{1,3}]].$$
\end{lemma}
\begin{proof}
Let $U=([0;(1,4)_{s},1,1,1,\overline{3,1}],[0;\underline{b}^s,1,2,\overline{3,1}])$ be the gap between $K(\underline{b}^s,1,1)$ and $K(\underline{b}^s,1,2)$. 

\vspace{0.1cm}

\textbf{Claim:} We have $|\hat{K}(\underline{c}^s,1,1)|>|U|$.

In fact, write $\hat{K}:=\hat{K}(\underline{c}^s,1,1)$
	\begin{equation}\label{K1}
	|\hat{K}|=\dfrac{[1;1,\overline{3,1}]-[1;1,\overline{1,3}]}{q^2_{\underline{c}^s}([1;1,\overline{3,1}]+[0;3,1,(4,1)_{s-1}])([1;1,\overline{1,3}]+[0;3,1,(4,1)_{s-1}])}
	\end{equation}
and
	$$|U|=\dfrac{[1;1,1,\overline{3,1}]-[1;2,\overline{3,1}]}{q^2_{\underline{b}^s}([1;1,1,\overline{3,1}]+[0;(4,1)_s])([1;2,\overline{3,1}]+[0;(4,1)_s])}.$$
	We have $$\dfrac{|\hat{K}|}{|U|}=\dfrac{q^2_{\underline{b}^s}}{q^2_{\underline{c}^s}}\cdot X\cdot Y,$$
	where 
	$$X=\dfrac{[1;1,\overline{3,1}]-[1;1,\overline{1,3}]}{[1;1,1,\overline{1,3}]-[1;2,\overline{3,1}]}=2$$	
	and
	$$Y=\dfrac{([1;1,1,\overline{3,1}]+[0;(4,1)_s])([1;2,\overline{3,1}]+[0;(4,1)_s])}{([1;1,\overline{3,1}]+[0;3,1,(4,1)_{s-1}])([1;1,\overline{1,3}]+[0;3,1,(4,1)_{s-1}])}>0.77.$$
	Since
	$$\dfrac{q_{\underline{b}^s}}{q_{\underline{c}^s}}=1+\dfrac{1}{3+[0;(1,4)_{s-1},1]}\ge 1.25$$
	we have 
	 $$\dfrac{|\hat{K}|}{|U|}>1.25^{2}\cdot 2\cdot 0.77>2.4.$$
	 In particular we proved the claim. Next we shall prove that the convex hull of $K'=K(\underline{b}^s,1,1)\cup K(\underline{b}^s,1,2)$ is greater than $|\hat{K}|=|\hat{K}((1,4)_{s-1},1,3,1,1)|.$ Note that the convex hull of $K(\underline{b}^s,1,1)\cup K(\underline{b}^s,1,2)$ has lenght
	 $$[0;\underline{b}^s,1,2,\overline{1,3}]-[0;\underline{b}^s,1,1,\overline{3,1}]$$
	 which is equals to
	 $$\dfrac{[1;1,\overline{3,1}]-[1;2,\overline{1,3}]}{q^2_{\underline{b}^s}([1;1,\overline{3,1}]+\beta_{\underline{b}^s})(1;2,\overline{1,3}]+\beta_{\underline{b}^s})}.$$
Therefore, using \ref{K1} we have
	$$\dfrac{|\hat{K}'|}{|\hat{K}|}=\dfrac{q^2_{\underline{c}^s}}{q^2_{\underline{b}^s}}\cdot X \cdot Y$$
where
	$$X=\dfrac{[1;1,\overline{3,1}]-[1;2,\overline{1,3}]}{[1;1,\overline{3,1}]-[1;1,\overline{1,3}]}>1,858$$	
and
	$$Y=\dfrac{([1;1,\overline{3,1}]+[0;3,1,(4,1)_{s-1}])([1;1,\overline{1,3}]+[0;3,1,(4,1)_{s-1}])}{([1;1,\overline{3,1}]+\beta_{\underline{b}^s})(1;2,\overline{1,3}]+\beta_{\underline{b}^s})}>1,17996$$
	where we used that $[0;3,1,(4,1)_{s-1}]>[0;3,1]$ and $\beta_{\underline{b}^s}=[0;(4,1)_s]<[0;\overline{4,1}].$
	Therefore,
	$$\dfrac{|\hat{K}'|}{|\hat{K}|}>0.624\cdot 1.858\cdot 1,17996>1.368.$$	
	In particular, $|\hat{K}'|$ is greater than any gap of $K(\underline{c}^s,1,1)$ and by Corollary \ref{c} we have that $W^{(s)}$ is an interval.

\end{proof}
\begin{lemma}\label{l.14}
We write $\underline{b}^s=(1,4)_s$ and $\underline{c}^s=(1,4)_{s-1},1,3.$ We have
	$$M(\underline{b}^s,1,2)+M(\underline{c}^s,1,1)>m(\underline{b}^s,1,1)+m(\underline{c}^s,1,2).$$
In particular, since $K(\underline{c}^s,1)=\cup_{j=1}^3 K(\underline{c}^s,1,j)$ we have that
	$$K(\underline{b}^s,1,1)\cup K(\underline{b}^s,1,2)+K(\underline{c}^s,1)$$
is the interval
	$$[[0;\underline{b}^s,1,1,\overline{3,1}]+[0;\underline{c}^s,1,\overline{1,3}],[0,\underline{b}^s,1,2,\overline{1,3}]+[0;\underline{c}^s,1,\overline{3,1}]].$$		
\end{lemma}

\begin{proof}
In order to prove that 
	$$M(\underline{b}^s,1,2)+M(\underline{c}^s,1,1)>m(\underline{b}^s,1,1)+m(\underline{c}^s,1,2).$$
it is enough to that
	$$A:=[0;\underline{b}^s,1,2,\overline{1,3}]-[0;\underline{b}^s,1,1,\overline{3,1}]>[0;\underline{c}^s,1,2,\overline{3,1}]-[0;\underline{c}^s,1,1,\overline{1,3}]=:B.$$
	Note that
	$$A=\dfrac{[1;1,\overline{3,1}]-[1;2,\overline{1,3}]}{q^2_{\underline{b}^s}([1;1,\overline{3,1}]+[0;(4,1)_s])([1;2,\overline{1,3}]+[0;(4,1)_s])}$$
	and
	$$B=\dfrac{[1;1,\overline{1,3}]-[1;2,\overline{3,1}]}{q^2_{\underline{c}^s}([1;1,\overline{1,3}]+[0;3,1,(4,1)_{s-1}])([1;2,\overline{3,1}]+[0;3,1,(4,1)_{s-1}])}.$$
	In particular,
	$$\dfrac{A}{B}=\dfrac{q^2_{\underline{c}^s}}{q^2_{\underline{b}^s}}\cdot X\cdot Y,$$
where
	$$X=\dfrac{[1;1,\overline{3,1}]-[1;2,\overline{1,3}]}{[1;1,\overline{1,3}]-[1;2,\overline{3,1}]}>3.716$$
and
	$$Y=\dfrac{([1;1,\overline{1,3}]+[0;3,1,(4,1)_{s-1}])([1;2,\overline{3,1}]+[0;3,1,(4,1)_{s-1}])}{([1;1,\overline{3,1}]+[0;(4,1)_s])([1;2,\overline{1,3}]+[0;(4,1)_s])}>0.977$$
	where we used that $[0;3,1,(4,1)_{s-1}]>[0;3,1]$ and $[0;(4,1)_s]<[0;\overline{4,1}]$.
	Therefore, using that $\dfrac{q^2_{\underline{c}^s}}{q^2_{\underline{b}^s}}>0.624$ by the inequality \ref{inequ.1},
	$$\dfrac{A}{B}>3.716\cdot 0.624\cdot 0,977>2.265.$$
	This finish the proof. The consequence is immediate.		
\end{proof}

Below we write $\tilde{K}(\underline{b})=\{[0;\underline{b},a_1,a_2,\tilde{\theta}];(a_1,a_2)\neq (1,3),  \ a_1a_2\tilde{\theta}\in D\}$ where we mean
$\tilde{K}=\{[0;a_1,a_2,\tilde{\theta}];(a_1,a_2)\neq (1,3) \ a_1a_2\tilde{\theta}\in D\}$ if $\underline{b}$ is the empty word.

\begin{proposition}
We have that
$K(1,3)+\tilde{K}(1,4)$ is the interval $[1.57041..., 1.61695...]$, where 
	$$1.57041...=[0,1,3,4,\overline{1,3}]+[0;1,4,4,\overline{1,3}] \ \mbox{and} \ 1.61695...=[0;\overline{1,3}]+[0;1,4,1,2,{1,3}].$$
More generally, for $s\ge 1$ we have $K(\underline{c}^s)+\tilde{K}(\underline{b}^s)$ is the interval
	$$[[0;\underline{c}^s,4,\overline{1,3}]+[0;\underline{b}^s,4,\overline{1,3}],[0;\underline{c}^s,\overline{1,3}]+[0;\underline{b}^s,1,2,\overline{1,3}]].$$
\end{proposition}

\begin{proof} Let $\underline{b}^s=\underline{b}^s$ and $(1,4)_{s-1},1,3=\underline{c}^s$. First of all we rewrite
	$$\tilde{K}=K(4)\cup K(3)\cup K(2)\cup K(1,1)\cup K(1,2).$$
Then
	$$\tilde{K}(\underline{b}^s)=\cup_{j=2}^4 K(\underline{b}^s,j) \cup K(\underline{b}^s,1,1)\cup K(\underline{b}^s,1,2).$$
	Then we have
	$$\tilde{K}(\underline{b}^s)+K(\underline{c}^s)=(\cup_{j=2}^4 K(\underline{b}^s,j)+K(\underline{c}^s))\cup (\cup_{p=1}^2 K(\underline{b}^s,1,p)+K(\underline{c}^s)).$$
	In a similar way we rewrite
	$$K(\underline{c}^s)=\cup_{m=2}^4 K(\underline{c}^s,m) \cup K(\underline{c}^s,1).$$	
	Then we have
	$$\cup_{p=1}^2 K(\underline{b}^s,1,p)+K(\underline{c}^s)=\cup_{j=1}^4(\cup_{p=1}^2 K(\underline{b}^s,1,p)+K(\underline{c}^s,j)).$$
	By Lemma \ref{l.14} we have that $\cup_{p=1}^2 K(\underline{b}^s,1,p)+K(\underline{c}^s,1)$ is an interval. By Lemma \ref{l36} $\cup_{p=1}^2 K(\underline{b}^s,1,p)+K(\underline{c}^s,j), j=2,3,4$ are intervals.	Note that 
	$$M(\underline{b}^s,1,2)+M(\underline{c}^s,j)>m(\underline{b}^s,1,1)+m(\underline{c}^s,j+1), \ j=1,2,3$$
	because $[0;\underline{b}^s,1,2,...]>[0;\underline{b}^s,1,1,...]$ and $[0;\underline{c}^s,j,...]>[0;\underline{c}^s,j+1,...]$ since $|\underline{b}^s|$ and $|\underline{c}^s|$ are even. This implies that
	$$\cup_{p=1}^2 K(\underline{b}^s,1,p)+K(\underline{c}^s)$$
	is the interval
	$$[[0;\underline{b}^s,1,1,\overline{3,1}]+[0;\underline{c}^s,4,\overline{1,3}],[0;\underline{b}^s,1,2,\overline{1,3}]+[0;\underline{c}^s,\overline{1,3}]].$$
	By Lemma \ref{l.11} $(\cup_{j=2}^4 K(\underline{b}^s,j)+K(\underline{c}^s))$ is the interval 
	$$[[0;\underline{b}^s,4,\overline{1,3}]+[0;\underline{c}^s,4,\overline{1,3}],[0;\underline{b}^s,2,\overline{3,1}]+[0;\underline{c}^s,\overline{1,3}]]$$ 
	Following the same lines as above we show that
	$$[0;\underline{b}^s,2,\overline{3,1}]+[0;\underline{c}^s,\overline{1,3}]>[0;\underline{b}^s,1,1,\overline{3,1}]+[0;\underline{c}^s,4,\overline{1,3}].$$  
	This shows that 
	$$\tilde{K}(\underline{b}^s)+K(\underline{c}^s)=(\cup_{j=2}^4 K(\underline{b}^s,j)+K(\underline{c}^s))\cup (\cup_{p=1}^2 K(\underline{b}^s,1,p)+K(\underline{c}^s))$$
	is the interval
	$$[[0;\underline{b}^s,4,\overline{1,3}]+[0;\underline{c}^s,4,\overline{1,3}],[0;\underline{b}^s,1,2,\overline{1,3}]+[0;\underline{c}^s,\overline{1,3}]].$$
\end{proof}

\begin{lemma} \label{Lemma}
For any $s\ge 1$ we have 
	\begin{equation}\label{Eq.gluing1}
	 [0;\underline{b}^s,1,2,\overline{1,3}]+[0;\underline{c}^s,\overline{1,3}]>2[0;\underline{b}^s,4,\overline{1,3}]
	\end{equation}
and
	\begin{equation} \label{Eq.gluing2}
	2[0;\underline{b}^s,\overline{1,3}]>[0;\underline{b}^{s+1},4,\overline{1,3}]+[0;\underline{c}^{s+1},4,\overline{1,3}]
	\end{equation}
	In particular, 
	\begin{equation}\label{gluing}
	\tilde{K}(\underline{b}^s)+K(\underline{c}^s)\cup K(\underline{b}^s)+K(\underline{b}^s)\cup \tilde{K}(\underline{b}^{s+1})+K(\underline{c}^{s+1})
	\end{equation}
	is the interval
	$$[[0;\underline{b}^s,4,\overline{1,3}]+[0;\underline{c}^s,4,\overline{1,3}],[0;\underline{b}^{s+1},1,2,\overline{1,3}]+[0;\underline{c}^{s+1},\overline{1,3}]]$$
\end{lemma}
\begin{proof} We begin proving (\ref{Eq.gluing1}). If $s=1$ we have 
	$$[0;1,4,1,2,\overline{1,3}]+[0;\overline{1,3}]>1.6169>1.6161>2\cdot [0;1,4,4,\overline{1,3}].$$
We suppose $s>1$.
	$$[0;\underline{b}^s,1,2,\overline{1,3}]-[0;\underline{b}^s,4,\overline{1,3}]=\dfrac{[4;\overline{1,3}]-[1;2,\overline{1,3}]}{q^2_{\underline{b}^s}([4;\overline{1,3}]+\beta_{\underline{b}^s})([1;2,\overline{1,3}]+\beta_{\underline{b}^s})}$$
and
	$$[0;\underline{b}^s,4,\overline{1,3}]-[0;\underline{c}^s,\overline{1,3}]=\dfrac{[1;\overline{3,1}]-[1;4,4,\overline{1,3}]}{q^2_{\underline{b}^{s-1}}([1;\overline{3,1}]+\beta_{\underline{b}^{s-1}})(1;4,4,\overline{1,3}]+\beta_{\underline{b}^{s-1}})}.$$
	Then,
	$$\dfrac{[0;\underline{b}^s,1,2,\overline{1,3}]-[0;\underline{b}^s,4,\overline{1,3}]}{[0;\underline{b}^s,4,\overline{1,3}]-[0;\underline{c}^s,\overline{1,3}]}=\dfrac{q^2_{\underline{b}^{s-1}}}{q^2_{\underline{b}^s}}\cdot X\cdot Y$$
	where
	$$X=\dfrac{[4;\overline{1,3}]-[1;2,\overline{1,3}]}{[1;\overline{3,1}]-[1;4,4,\overline{1,3}]}>131$$
	and
	$$Y=\dfrac{([1;\overline{3,1}]+\beta_{\underline{b}}^{s-1})(1;4,4,\overline{1,3}]+\beta_{\underline{b}^{s-1}})}{([4;\overline{1,3}]+\beta_{\underline{b}^s})([1;2,\overline{1,3}]+\beta_{\underline{b}^s})}>0.268$$
	using that $[0;4,1]\le [0;(4,1)_{s-1}]<[0;\overline{4,1}]$, if $s>1$.
	
	Note that $q_{\underline{b}^s}=4q_{\underline{b}^{s-1}1}+q_{\underline{b}^{s-1}}$ and then 
	$$\dfrac{q_{\underline{b}^s}}{q_{\underline{b}^{s-1}}}=\dfrac{4}{\beta_{\underline{b}^{s-1}1}}+1<\dfrac{4}{[0;\overline{1,4}]}+1$$
	where we used that $\beta_{\underline{b}^{s-1}1}=[0;(1,4)_{s-1},1]>[0;\overline{1,4}].$	
	In particular, 
	$$\dfrac{q_{\underline{b}^{s-1}}}{q_{\underline{b}^{s}}}>0.1715.$$
	This implies
	$$\dfrac{[0;\underline{b}^s,1,2,\overline{1,3}]-[0;\underline{b}^s,4,\overline{1,3}]}{[0;\underline{b}^s,4,\overline{1,3}]-[0;\underline{c}^s,\overline{1,3}]}>(0.1715)^2\cdot 131\cdot 0.268>1.03.$$
	The proof of (\ref{Eq.gluing2}) is similar. The consequence (\ref{gluing}) is now immediate.
\end{proof}

\begin{proposition}\label{p1}
We have 
	$$4+\bigcup_{s\ge 1} \tilde{K}(\underline{b}^s)+K(\underline{c}^s)\cup K(\underline{b}^s)+K(\underline{b}^s)\supset [1+\sqrt{21},\sqrt{32}).$$
\end{proposition}
\begin{proof}
This is an immediate consequence of Lemma \ref{Lemma} and using the fact that $4+[0;\underline{b}^s,\theta]+[0;\underline{b}^s,\tilde{\theta}]\to 4+2[0;\overline{1,4}]=\sqrt{32}$ when $s\to \infty$ for any $\theta,\tilde{\theta}\in \mathbb{N}^{\mathbb{N}}.$
\end{proof}

To prove theorem B let us consider $A_N=[0;\overline{N,1}]$ and $B_N=[0;\overline{1,N}]$. We have that 
$$NA_N+a_NB_N=1 \ \mbox{and} \ B_N+B_NA_N=1,$$ 
so
$A_N=\dfrac{B_N}{N}.$

Thus, we have

$$B_N=\frac{-N+\sqrt{N^2+4N}}{2} \ \mbox{and} \ A_N=\frac{-N+\sqrt{N^2+4N}}{2N}.$$ 
Recall example \ref{ex1}, where $C(N)=\{x=[0;a_1,a_2,...]; a_i\le N, \forall i\ge 1\}$ and $g$ is the Gauss map. Next, consider $C(\underline{b},k)=\{x=[0;\underline{b},\theta]; \underline{b}\in \{1,2,...,n\}^{s}, \ \theta\in \{1,2,...,k\}^{\mathbb{N}}\}$ 
\begin{theorem}\label{theorem4.8}
 For any $\underline{b}\in \{1,2,...,k\}^s, \ s\ge 0,  \ 4\le k\le n$ we have $\tau(C(\underline{b},k))>1$.
\end{theorem}
\begin{proof}
We suppose $s$ even. If $|\underline{a}|$ is even we divide $I_{\underline{ba}}=[[0;\underline{ba},\overline{k,1}],[0;\underline{ba},\overline{1,k}]]$ in $k$ intervals $I_{\underline{ba},j}=[[0;\underline{ba},j,\overline{1,k}],[0;\underline{ba},j,\overline{k,1}]]$ and we write 
	$$O^{j}_{\underline{ba}}([0;\underline{ba},j+1,\overline{k,1}],[0;\underline{ba},j,\overline{1,k}])$$
for $j=1,...,k-1$ the $k-1$ gaps between $I_{\underline{ba},j+1}$ and $I_{\underline{ba},j}.$ We observe that
	$$I_{\underline{ba}}=I_{\underline{ba},k}\cup O^{k-1}_{\underline{ba}}\cup I_{\underline{ba},k-1}\cup O^{k-2}_{\underline{ba}}\cup ...\cup O^1_{\underline{ba}}\cup I_{\underline{ba},1}, \ I_{\underline{ba},j}<O^{j-1}_{\underline{ba}}<I_{\underline{ba},j-1}$$
\vspace{0.1cm}
	\begin{center}
\begin{tikzpicture}[domain=0:0.5,xscale=3,yscale=3]
\draw (-2.3,0.3) -- (1.3,0.3);
\node at (-2.3,0.4) {$I_{\underline{ba}}$};
\draw (-2.3,0) -- (-1.9,0);
\node[below] at (-2.2,-0.1) {$I_{\underline{ba},k}$};
\node[red] at (-1.9,0) { ( };
\node[red] at (-1.6,0) { ) };
\node[below] at (-1.75,-0.1) {$O^{k-1}_{\underline{ba}}$};

\draw (-1.6,0) -- (-1.1,0);
\node[below] at (-1.35,-0.1) {$I_{\underline{ba},k-1}$};

\node[red] at (-1.1,0) { ( };
\node[red] at (-0.7,0) { ) };

\draw (-0.7,0)--(-0.1,0);
\node at (0,0) { ... };

\node[red] at (0.1,0) { ( };
\node[red] at (0.6,0) { ) };
\node[below] at (0.35,-0.1) {$O^1_{\underline{ba}}$};
\node[below] at (0.9,-0.1) {$I_{\underline{ba},1}$};
\draw (0.6,0) -- (1.3,0);

\end{tikzpicture}
	\end{center}
	In a similar way as section \ref{section3} we can prove that $|O^j_{\underline{ba}}|<|O^i_{\underline{ba}}|$ if $j>i$ and $|O^1_{\underline{ba}}|<|O^{k-1}_{\underline{ba^{\ast}}}|$. Moreover, to calculate the thickness of $C(\underline{b},k)$ we must just to calculate $|I_{\underline{ba},k}|/|O^{k-1}_{\underline{ba}}|.$ There is no difficult to see that
		$$\dfrac{|I_{\underline{ba},k}|}{|O^{k-1}_{\underline{ba}}|}=\dfrac{[0;\overline{1,k}]-[0;\overline{k,1}]}{[1;\overline{k,1}]-[0;\overline{1,k}]}\cdot \dfrac{[k-1;\overline{1,k}]+\beta_{\underline{ba}}}{[k;\overline{k,1}]+\beta_{\underline{ba}}}.$$
		Since $F(k)=\dfrac{[0;\overline{1,k}]-[0;\overline{k,1}]}{[1;\overline{k,1}]-[0;\overline{1,k}]}\ge F(4)>1.6$ and $G(k)\ge \dfrac{[k-1;\overline{1,k}]+\beta_{\underline{ba}}}{[k;\overline{k,1}]+\beta_{\underline{ba}}}>0.9$ we have
		$$\dfrac{|I_{\underline{ba},k}|}{|O^{k-1}_{\underline{ba}}|}>1.$$
	The other cases are purely analogous. 	
\end{proof}
\begin{remark}
Note that the above proof give us precisely the thickness of $C(\underline{b},k),$ we must take the infimum in $\underline{a}$.
\end{remark}

Using the above characterizations we prove	
\begin{theorem}\label{T4.8}
For any $s\ge 1$ we have $4+\tilde{K}(\underline{b}^s)+K(\underline{c}^s)\subset L_f(\Lambda_4)$ and $4+K(\underline{b}^s)+K(\underline{b}^s)\subset L_f(\Lambda_4).$
\end{theorem}
\begin{proof}
Take $\ell\in 4+\tilde{K}(\underline{b}^s)+K(\underline{c}^s).$ There are $\tilde{\theta}=(a_n)_{n>2s}\in D$ and $\theta=(a_{-n})_{n>2s}\in D$ with $(a_{-2s-1},a_{-2s-2})\neq (1,3)$ such that 
	$$\ell=4+[0;\underline{b}^s,\tilde{\theta}]+[0;\underline{c}^s,\theta].$$ 
	Consider $\alpha=...a_{-2s-3},a_{-2s-2},a_{-2s-1},3,1,(4,1)_{s-1},4^*,(1,4)_s,a_{2s+1},a_{2s+2},a_{2s+3},...$, we just write $\alpha=(a_n)_{n\in \mathbb{Z}}$ with $(a_{-2s},...,a_{-1},a_0,a_1,...,a_{2s})=(3,1,(4,1)_{s-1},4^*,(1,4)_s)$. 
	
	\textbf{Claim 1:} $m(\alpha)=\sup\{\lambda_k(\alpha); k\in \mathbb{Z}\}=\lambda_0(\alpha)=\ell.$
	
	\vspace{0.3cm}
	Firstly we note that if $|j|>2s$ since $(1,4)$ does not occur in $\tilde{\theta}$ and $\theta$ we have $\lambda_j(\alpha)=[a_j;a_{j+1},...]+[0;a_{j-1},a_{j-2},...]\le 4+2[0;\overline{1,3}]<\lambda_0(\alpha).$ On the other hand if $j\le 2s$ is odd then $\lambda_j(\alpha)=1+[0;4,...]+[0;4,...]<2<\lambda_0(\alpha).$ Then we study only $\lambda_{2j}(\alpha)$ with $j\le 2s$. In this case
	$\lambda_{2j}(\alpha)=4+[0;(1,4)_{s-j},\theta]+[0;(1,4)_{s+j-1},1,3,\tilde{\theta}]$ and $\lambda_0(\alpha)=4+[0;(1,4)_s,\theta]+[0;(1,4)_{s-1},1,3,\tilde{\theta}].$ Since $\theta\neq (1,3,...)$ and $(1,4)$ does not occur we have 
	$$\lambda_{2j}(\alpha)\le 4+[0;(1,4)_{s-j},1,2,\overline{1,3}]+[0;(1,4)_{s+j-1},\overline{1,3}]=:C+D.$$
	On the other hand
	$$\lambda_0(\alpha)\ge 4+[0;(1,4)_s,4,\overline{1,3}]+[0;(1,4)_{s-1},1,3,4,\overline{1,3}]=:A+B.$$
	We shall prove that if $j\ne 0$ then $A+B>C+D$. In fact,
	$$A-C=\dfrac{[1;2,\overline{1,3}]-[1;4,(1,4)_{j-1},4,\overline{1,3}]}{q^2_{\underline{b}^{s-j}}([1;2,\overline{1,3}]+\beta_{\underline{b}^{s-j}})([1;4,(1,4)_{j-1},4,\overline{1,3}]+\beta_{\underline{b}^{s-j}})}$$
	and
	$$D-B=\dfrac{[1;3,4,\overline{1,3}]-[1;(4,1)_{j-1},4,\overline{1,3}]}{q^2_{\underline{b}^{s-1}}([1;3,4,\overline{1,3}]+\beta_{\underline{b}^{s-1}})([1;(4,1)_{j-1},4,\overline{1,3}]+\beta_{\underline{b}^{s-1}})}.$$
	Since $q_{\underline{b}^{s-j}}$ is decreasing with $j$ we have $A-C$ is maximum with $j=1$. Then we take $j=1$. Therefore,
	$$\dfrac{A-C}{D-B}=\dfrac{[1;2,\overline{1,3}]-[1;4,4,\overline{1,3}]}{[1;3,4,\overline{1,3}]-[1;4,\overline{1,3}]}\cdot Y >1.172\cdot Y$$
	where $$Y=\dfrac{([1;3,4,\overline{1,3}]+\beta_{\underline{b}^{s-1}})([1;4,\overline{1,3}]+\beta_{\underline{b}^{s-1}})}{([1;2,\overline{1,3}]+\beta_{\underline{b}^{s-1}})([1;4,4,\overline{1,3}]+\beta_{\underline{b}^{s-1}})}>0.943$$
	where the minimum is attained with $s=1$.
	It follows that 
	$$\dfrac{A-C}{D-B}>1.172\cdot 0.943>1.105$$
	and we have the claim 1.
	
	Note that claim 1 implies that $\ell\in M_f(\Lambda_4).$ Next, let us consider 
	\begin{equation}\label{seq}
	\underline{\theta}^m=(a_{-m},a_{-m+1},...,a_{-1},a_0,a_1,...,a_{m-1},a_m), m\in \mathbb{N}.
	\end{equation}
	 If $\gamma=(\underline{\theta}^{|m|})_{m\in \mathbb{Z}}$ we can see, by claim 1, that $\tilde{f}(\sigma^r(\gamma))\le \tilde{f}(\alpha), \ r\in \mathbb{Z}$. On the other hand we can find a sequence $x_n\in \mathbb{N}$ such that 
	$$\lim_{n\to +\infty}\tilde{f}(\sigma^{x_n}(\gamma))=\tilde{f}(\alpha)=\lambda_0(\alpha)=\ell,$$
	$x_n=n(n+1)$ for instance.
	To finish the proof take $\ell\in 4+K(\underline{b}^s)+K(\underline{b}^s)$. That is, $\ell=4+[0;\underline{b}^s,\theta^+]+[0;\underline{b}^{s},\theta^-]$. We define $\alpha=(a_n)_{n\in \mathbb{Z}}$ where $(a_{m})_{m=-2s}^{2s}=((4,1)_s,4^*,(1,4)_s).$ If $|k|\ge 2s$ then $\lambda_{0}(\alpha)>\lambda_k(\alpha)$ because $\theta^-$ and $\theta^+$ are in $D$. This proves that $m(\alpha)\in \{\lambda_i(\alpha); |i|< 2s\}$. If $k$ is odd then $\lambda_k(\alpha)=1+[0;4,1,...]+[0;4,1,...]<\lambda_0(\alpha)$. Then we suppose $k=2\tilde{k}$. It is enough to show that

	\textbf{Claim 2:} $m(\alpha)=\sup\{\lambda_k(\alpha); k\in \mathbb{Z}\}=\lambda_0(\alpha)=\ell.$
	
	It is equivalent to show that, if $\tilde{k}\neq 0$ then $[0;(1,4)_{s},\theta^+]+[0;(1,4)_{s},\theta^-]>[0;(1,4)_{s-\tilde{k}},\theta^+]+[0;(1,4)_{s+\tilde{k}},\theta^-]$.
	Note that it is equivalent to show that 
	$$A:=[0;(1,4)_{s},\theta^+]-[0;(1,4)_{s-\tilde{k}},\theta^+]>[0;(1,4)_{s+\tilde{k}},\theta^-]-[0;(1,4)_{s},\theta^-]=:B.$$
	Suppose $\theta^+=(a_0,a_1,...)$ and $\theta^-=(c_0,c_1,...).$ Let $\underline{d}=(1,4)_{s-\tilde{k}}$. We have
	$$A=\dfrac{\vert [a_0;a_1,...]-[1;(4,1)_{\tilde{k}-1},4,a_0,a_1,...]\vert}{q^2_{\underline{d}}([a_0;a_1,...]+\beta_{\underline{d}})([1;(4,1)_{\tilde{k}-1},4,a_0,a_1,...])}$$
	and
	$$B=\dfrac{\vert [c_0;c_1,...]-[1;(4,1)_{\tilde{k}-1},4,c_0,c_1,...]\vert}{([c_0;c_1,...]+\beta_{\underline{b}})([1;(4,1)_{\tilde{k}-1},4,c_0,c_1,...]+\beta_{\underline{b}})}.$$
	
	Consider $f,g:[[1;\overline{3,1}],[4;\overline{1,3}]]\rightarrow \mathbb{R}$ given by
	$$f(x)=\dfrac{x-[1;4,x]}{(x+\beta_{\underline{b}})([1;4,x]+\beta_{\underline{b}})}\ge \dfrac{x-[1;4,x]}{1.5\cdot (x+0.21)}>0.05$$
	and
	$$g(x)=\dfrac{x-[1;4,x]}{(x+\beta_{\underline{d}})([1;\overline{4,1}]+\beta_{\underline{d}})}\le \dfrac{x-[1;4,x]}{1.4\cdot (x+\beta_{\underline{d}})}\le 1.0072.$$
	Therefore,
	$$\dfrac{A}{B}>\dfrac{q^2_{\underline{b}}}{q^2_{\underline{d}}}\cdot 0.04.$$ But since $\tilde{k}\neq 0$ by the Euler's rule we have
	$$\dfrac{q^2_{\underline{b}}}{q^2_{\underline{d}}}>33.$$
	Then, $\dfrac{A}{B}>1$.
	
The claim 2 implies that $\lambda_0(\alpha)>\lambda_i(\alpha)$ for $0\neq |i|\le 2s$. Therefore, 
	$$m(\alpha)=\sup_{k\in \mathbb{Z}}\lambda_k(\alpha)=\lambda_0(\alpha)=\ell.$$
Using (\ref{seq}) we have that $\ell\in L_f(\Lambda_4).$

\end{proof}

Next, for $\underline{b}\in \{1,2,...,k\}^t$ we write
	$$\tilde{C}(\underline{b},k)=\{x=[0;\underline{b},a_1,a_2,\theta]; (a_1,a_2)\neq (1,k-1), \ \theta=(a_n)_{n\ge 3}\in \{1,2,...,k-1\}\}.$$
\begin{proposition}
Let $\underline{b}^s=(1,k)_s$, $s\ge 0$. We have that $\tau(\tilde{C}(\underline{b}^s,k))>1$, if $k\ge 5$.
\end{proposition}	
\begin{proof}
Write $I_{\underline{b}^s}=[[0;\underline{b}^s,\overline{k-1,1}],[0;\underline{b}^s,\overline{1,k-1}]]$, $I_{\underline{b}^sj}=[[0;\underline{b}^sj,\overline{1,k-1}],[0;\underline{b}^sj,\overline{1,k-1}]]$ and note that $C(\underline{b}^sj,k)=I_{\underline{b}^sj}\cap C(\underline{b}^s,k).$
By definition of $\tilde{C}(\underline{b}^s,k)$ we have that
\begin{equation}
\tilde{C}(\underline{b}^s,k)=\bigcup_{j=2}^{k-1} C(\underline{b}^sj,k)\cup \bigcup_{r=1}^{k-2} C(\underline{b}^s1r,k)
\end{equation}
because $(1,k-1)$ does not appear in $\tilde{C}(\underline{b}^s,k).$
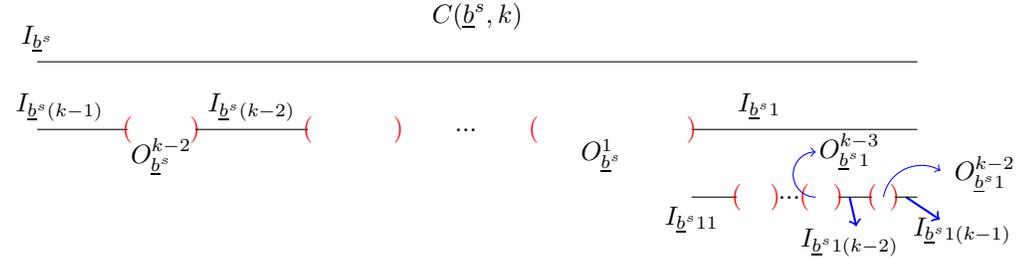
\vspace{0.1cm}
	\begin{center}
\begin{tikzpicture}[domain=0:0.5,xscale=3,yscale=3]
\draw (-2.3,0.3) -- (1.6,0.3);
\node at (-2.3,0.4) {$I_{\underline{b}^s}$};
\node at (-0.35,0.5) {$C({\underline{b}^s},k)$};

\draw (-2.3,0) -- (-1.9,-0);
\node[below] at (-2.2,0.2) {$I_{\underline{b}^s(k-1)}$};
\node[red] at (-1.9,0) { ( };
\node[red] at (-1.6,0) { ) };
\node[below] at (-1.75,0) {$O^{k-2}_{\underline{b}^s}$};

\draw (-1.6,0) -- (-1.1,0);
\node[below] at (-1.35,0.2) {$I_{\underline{b}^s(k-2)}$};
\node[red] at (-1.1,-0) { ( };
\node[red] at (-0.7,-0) { ) };

\node at (-0.4,0) {...};

\draw (0.6,0)--(1.6,0);
\node[red] at (-0.1,0) { ( };
\node[red] at (0.6,0) { ) };
\node[below] at (0.2,0) {$O^1_{\underline{b}^s}$};
\node[below] at (0.9,0.2) {$I_{\underline{b}^s1}$};

\draw (0.6,-0.3)--(0.8,-0.3);
\node[red] at (0.8,-0.3) {$ ( $};
\node[red] at (0.97,-0.3) {$ ) $};
\node at (1.033,-0.3) {$ ... $};

\node at (0.6,-0.4){$I_{\underline{b}^s11}$};

\draw (1.25,-0.3)--(1.4,-0.3); 
\node[red] at (1.4,-0.3) {$ ( $};
\node[red] at (1.5,-0.3) {$ ) $};
\draw[thick, blue, ->] (1.3,-0.3) -- (1.33,-0.43);
\draw[thick, blue, ->] (1.55,-0.3) -- (1.7,-0.4);


\draw [->, blue] (1.45,-0.3) arc [radius=0.17, start angle=170, end angle=60];
\node at (1.9,-0.2) {$O^{k-2}_{\underline{b}^s1}$};
\node at (1.3,-0.5) {$I_{\underline{b}^s1(k-2)} $};
\node at (1.8,-0.45){$I_{\underline{b}^s1(k-1)}$};
\node[red] at (1.25,-0.3) {$ ) $};
\node[red] at (1.1,-0.3) {$ ( $};
\draw [->, blue] (1.15,-0.3) arc [radius=0.1, start angle=270, end angle=90];
\node at (1.3,-0.1) {$O^{k-3}_{\underline{b}^s1}$};

\draw (1.5,-0.3)--(1.6,-0.3);


\end{tikzpicture}
	\end{center}
We only need to verify that $|I_{\underline{b}^s(k-2)}|>|O^{k-3}_{\underline{b}^s1}|, \ k\ge 5$ because we know that the Cantor sets of type $C(\underline{d},k)$ have thickness greater than $1$. For this sake as usually we write
	$$|I_{\underline{b}^s1(k-2)}|=[0;\underline{b}^s,1,k-2,\overline{1,k-1}]-[0;\underline{b}^s,1,k-2,\overline{k-1,1}]$$
	and
	$$|O^{k-3}_{\underline{b}^s1}|=[0;\underline{b}^s,1,k-2,\overline{k-1,1}]-[0;\underline{b}^s,1,k-3,\overline{1,k-1}].$$
	We have
	$$|I_{\underline{b}^s1(k-2)}|=\dfrac{[1;k-2,\overline{k-1,1}]-[1;k-2,\overline{1,k-1}]}{q^2_{\underline{b}^s}([1;k-2,\overline{k-1,1}]+\beta_{\underline{b}^s})([1;k-2,\overline{1,k-1}]+\beta_{\underline{b}^s})}$$
	and
	$$|O^{k-3}_{\underline{b}^s1}|=\dfrac{[1;k-3,\overline{1,k-1}]-[1;k-2,\overline{k-1,1}]}{q^2_{\underline{b}^s}([1;k-3,\overline{1,k-1}]+\beta_{\underline{b}^s})([1;k-2,\overline{k-1,1}]+\beta_{\underline{b}^s})}.$$
	Let $X=\dfrac{[1;k-2,\overline{k-1,1}]-[1;k-2,\overline{1,k-1}]}{[1;k-3,\overline{1,k-1}]-[1;k-2,\overline{k-1,1}]}$ and $Y=\dfrac{[1;k-3,\overline{1,k-1}]+\beta_{\underline{b}^s}}{[1;k-2,\overline{k-1,1}]+\beta_{\underline{b}^s}}.$
	Since $[1;k-3,...]>[1;k-2,...]$ we have that $Y>1$. On the other hand, there is no difficult to see that for $k\ge 5$ we have $X>1$. This shows that $|I_{\underline{b}^s(k-2)}|>|O^{k-3}_{\underline{b}^s1}|, \ k\ge 5$ and we are done.
	
\end{proof}
It follows from the above Lemma and Corollary \ref{c} that $\tilde{C}(\underline{b}^s,k)+C(\underline{c}^s,k)$ is the interval
	$$[[0;\underline{b}^{s},\overline{k-1,1}]+[0;\underline{c}^{s},\overline{k-1,1}],[0;\underline{b}^s,1,k-2,\overline{1,k-1}]+[0;\underline{c}^s,\overline{1,k-1}]].$$
In the same way $C(\underline{b}^s,k)+C(\underline{b}^s,k)$ is the interval
	$$[2\cdot [0;\underline{b}^{s},\overline{k-1,1}],2\cdot [0;\underline{b}^s,\overline{1,k-1}]].$$
\begin{proposition}\label{prop4}
Let $\underline{b}^s=(1,k)_s$ and $\underline{c}^s=((1,k)_{s-1},1,k-1)$. We have
	$$k+\bigcup_{s=1}^{\infty} (\tilde{C}(\underline{b}^s,k)+C(\underline{c}^s,k)\cup C(\underline{b}^s,k)+C(\underline{b}^s,k))\supset [1+\sqrt{k^2+2k-3},\sqrt{k^2+4k}).$$
\end{proposition}
\begin{proof}
In the same way as Lemma (\ref{Lemma}) implies the Proposition (\ref{p1}) we need to prove that 
	$$[0;\underline{b}^s,1,k-2,\overline{1,k-1}]+[0;\underline{c}^s,\overline{1,k-1}]>2\cdot [0;\underline{b}^s,\overline{k-1,1}]$$
and
	$$2\cdot [0;\underline{b}^s,\overline{1,k-1}]>[0;\underline{b}^{s+1},\overline{k-1,1}]+[0;\underline{c}^{s+1},\overline{k-1,1}].$$
	The proof follows the same lines as Lemma (\ref{Lemma}).
\end{proof}

  Now we prove the analogous to the Theorem \ref{T4.8}
\begin{theorem}\label{T4.11}
For any $s\ge 0, k\ge 5$ we have 
	$$k+\tilde{C}(\underline{b}^s,k)+C(\underline{c}^s,k)\subset L_f(\Lambda_k)\subset M_f(\Lambda_k)$$ 
$$k+C(\underline{b}^s,k)+C(\underline{b}^s,k)\subset L_f(\Lambda_k)\subset M_f(\Lambda_k),$$
where $\underline{b}^s=( )$ is the empty word if $s=0$.	

\end{theorem}
\begin{proof}
We prove the case $s\ge 1$, the case $s=0$ follows the same lines. We take $\ell=k+[0;\underline{b}^s,\theta]+[0;\underline{c}^s,\tilde{\theta}]$ and $\alpha=(\underline{b}^s\theta)^t k^*\underline{c}^s\tilde{\theta}$, where the asterisk is the $0$-th position, in the same way as theorem \ref{T4.8} we can prove that $\lambda_0(\alpha)>\sup\{\lambda_{i}(\alpha); i\neq 0\}$, the hardest case is to show that $\lambda_0(\alpha)>\lambda_2(\alpha)$ and we shall prove this. Note that in particular this proves that $\ell\in M_f(\Lambda_k)$. For this sake write
$$\ell=\lambda_0(\alpha)=k+[0;\underline{b}^s,\theta]+[0;\underline{c}^s,\tilde{\theta}]:=A(k)+B(k)$$
and
$$\lambda_2(\alpha)=k+[0;\underline{b}^{s-1},\theta]+[0;\underline{c}^{s+1},\tilde{\theta}]:=C(k)+D(k).$$
If we write $z=[\theta], \tilde{z}=[\tilde{\theta}]$ we have that
	$$F_k(z)=A(k)-C(k)=\dfrac{z-[1;k,z]}{q^2_{\underline{b}^{s-1}}(z+\beta_{\underline{b}^{s-1}})([1;k,z]+\beta_{\underline{b}^{s-1}})}$$
and 
	$$G_k(z)=D(k)-B(k)=\dfrac{[1;k-1,\tilde{z}]-[1;k,\tilde{z}]}{q^2_{\underline{b}^{s-1}}([1;k-1,\tilde{z}]+\beta_{\underline{b}^{s-1}})([1;k,\tilde{z}]+\beta_{\underline{b}^{s-1}})}.$$
	Since $\theta=(a_n)_n$ is such that $(a_1,a_2)\neq (1,k-1)$ we have that 
	$$F_k: [[1;k-2,\overline{1,k-1}],[k-1;\overline{1,k-1}]\rightarrow \mathbb{R} \ \mbox{and} \ G_k: [[1;\overline{k-1,1}],[k-1;\overline{1,k-1}]]\rightarrow \mathbb{R}.$$ 
	Both functions $F_k$ and $G_k$ are increasing and then 
	$$\dfrac{A(k)-C(k)}{D(k)-B(k)}\ge \dfrac{F_k([1;k-2,\overline{1,k-1}])}{G_k([k-1;\overline{1,k-1}])}.$$
	Now, we shall prove that 
	\begin{equation}\label{inequality2}
	\dfrac{F_k([1;k-2,\overline{1,k-1}])}{G_k([k-1;\overline{1,k-1}])}>1, \ k\ge 5.
	\end{equation}
In order to do that, write 
	$$X=\dfrac{[1;k-2,\overline{1,k-1}]-[1;k,1,k-2,\overline{1,k-1}]}{[1;k-1,\overline{k-1,1}]-[1;k,\overline{k-1,1}]}$$
and
	$$Y=\dfrac{([1;k-1,\overline{k-1,1}]+[0;(k,1)_{s-1}])([1;k,\overline{k-1,1}]+[0;(k,1)_{s-1}])}{([1;k-2,\overline{1,k-1}]+[0;(k,1)_{s-1}])([1;k,1,k-2,\overline{1,k-1}]+[0;(k,1)_{s-1}])}.$$		
	We shall prove that $X>\dfrac{3}{2}$ and $Y>\dfrac{2}{3}$, and since $\dfrac{F_k([1;k-2,\overline{1,k-1}])}{G_k([k-1;\overline{1,k-1}])}=X\cdot Y$ we have that (\ref{inequality2}) is true.
	We have the following inequalities
	\begin{equation}\label{ineq 1, T4.11}
	[1;k-2,\overline{1,k-1}]>[1;k-2,1,k-1,1]=1+\dfrac{k+1}{k^2-2},
	\end{equation}
	\begin{equation}\label{ineq 2, T4.11}
	[1;k,1,k-2,\overline{1,k-1}]<[1;k,1,k-2]=1+\dfrac{k-1}{k^2-2}
	\end{equation}
	\begin{equation}\label{ineq 3, T4.11}
	[1;k-1,\overline{k-1,1}]<[1;k-1,k-1,1]=1+\dfrac{k}{k^2-k+1}
	\end{equation}
	\begin{equation}\label{ineq 4, T4.11}
	[1;k,\overline{k-1,1}]>[1;k,k-1]=1+\dfrac{k-1}{k^2-k+1}.
	\end{equation}
	Using the inequalities (\ref{ineq 1, T4.11})--(\ref{ineq 4, T4.11}) we have
	$$X>\dfrac{\dfrac{2}{k^2-2}}{\dfrac{1}{k^2-k+1}}=\dfrac{2\cdot (k^2-k+1)}{k^2-2}>\dfrac{3}{2}.$$
	Using the same idea we can show in a analogous way that $Y>\dfrac{2}{3}.$ This shows that $m(\alpha)=\lambda_0(\alpha)=\ell$ and we proceed to show that $\ell\in L_f(\Lambda_k)$ using \ref{seq}. The other case is analogous.
\end{proof}
Finally we are ready to prove our main theorem
\begin{mydef4}There exists a $C^{\infty}$-conservative diffeomorphism $\varphi: \mathbb{S}^2\rightarrow \mathbb{S}^2$ from the sphere and a $C^{\infty}$-map $f:\mathbb{S}^2\rightarrow \mathbb{R}$ for which for any $k>1$ and $k\neq 3$ there is a horseshoe $\Lambda_k$ such that 
$$L_f(\Lambda_k)=L\cap (-\infty, \sqrt{k^2+4k}] \quad \mbox{and} \quad M_f(\Lambda_k)=M\cap (-\infty,\sqrt{k^2+4k}].$$
\end{mydef4}
\begin{proof} The diffeomorphism $\varphi$ is given by (\ref{dif}), the horseshoe is $\Lambda_k = C(k)\times \tilde{C}(k)$ and $f(x,y)=x+y$. We know that ( see \cite{CF}, pag 58 ) $L\cap (-\infty,\sqrt{12}]=L_f(\Lambda_2)$. Thus the theorem is true for $k=1,2$. 

Why the above theorem does not work for $k=3$? It is well known, see for instance \cite{CF}, that $[c_F,+\infty)\subset L$ where $c_F=4.527...$ is the Freiman's constant. On the other hand, by theorem 1.0.11 by S. Astels in \cite{A}, writing $a=[0;\overline{1,3}]+[0;1,3,1,2,\overline{1,3}]$ and $b=2[0;1,3,1,3,\overline{3,1}]$ we have that $(a,b)\cap C(3)+C(3)=\emptyset$. In particular, this implies that
	$$(3+a,3+b)\cap f(\Lambda_3)=\emptyset.$$
Since $(3+a,3+b)\subset L\cap (-\infty, \sqrt{21}]$ we have that 
	$$L_f(\Lambda_3)\neq L\cap (-\infty, \sqrt{21}].$$

First we prove the case $k=4$, that is,
	$$L_f(\Lambda_4)=L\cap (-\infty,\sqrt{32}].$$
	
	Since $L_f(\Lambda_4)\subset  f(\Lambda_4)$ and $\max f(\Lambda_4)=\sqrt{32}$ we have $L_f(\Lambda_4)\subset L\cap (-\infty,\sqrt{32}].$ We shall prove the converse inclusion.
	
\textbf{Case 1: $\ell\in L\cap (-\infty, 5.34]$}
	
Let $\underline{\theta}\in \mathbb{N}^{\mathbb{N}}$ such that $\ell=\ell(\underline{\theta})$. Consider $\underline{\theta}=(a_n)_{n\in \mathbb{Z}}\in \mathbb{N}^{\mathbb{Z}}$ and set $N_j(\underline{\theta})=\{n\in \mathbb{N}; a_n\ge j\}$. It is easy to see that if $N_j(\underline{\theta})$ is infinite then
	$$\ell(\underline{\theta})\ge j+2A_j$$
	Therefore, if $\underline{\theta}$ is such that $N_5(\underline{\theta})$ is infinite we have that $\ell(\underline{\theta})\ge \frac{20+\sqrt{45}}{5}=5.3416407865....$ In particular, if $\ell\le 5.34$ we can suppose that $\underline{\theta}\in \{1,2,3,4\}^{\mathbb{N}}.$
	This implies that 
	$$L_f(\Lambda_4)\cap (-\infty, 5.34]=L\cap (-\infty, 5.34].$$
	
	\textbf{Case 2: $5.34<\ell\le 1+\sqrt{21}$}
	
If we take $\ell\in L\cap (-\infty, 1+\sqrt{21}]$ such that $\ell>5.34$, by the theorem 5, pag 53, chapter 4 of \cite{CF}, we have that $\ell=4+[0;a_1,a_2,...]+[0;a_{-1},a_{-2},...]$, where $(a_i,a_{i+1})\notin \{(1,4),(2,4)\}$, $i\in \mathbb{N}$ and $(a_{-i},a_{-(i+1)})\notin \{(1,4),(2,4)\}$, $i\in \mathbb{N}$. Moreover, $\tilde{f}(\sigma^k(\alpha))\le \tilde{f}(\alpha)$, and using a sequence like (\ref{seq}) we have that $L_f(\Lambda_4)$. 

\textbf{Case 3: $\ell\in L\cap (1+\sqrt{21},\sqrt{32}]$}

In this case there exists $\alpha\in \{1,2,3,4\}^{\mathbb{Z}}$ such that $\ell=\ell_f(\alpha)$. In fact, by Proposition \ref{p1}, there is $s\ge 1$ such that either 
\begin{itemize}
\item $\ell\in 4+\tilde{K}(\underline{b}^s)+K(\underline{c}^s)$ or

\item $\ell\in 4+K(\underline{b}^s)+K(\underline{b}^s)$
\end{itemize}
Since, by Theorem \ref{T4.8} we have $\ell\in L_f(\Lambda_4).$ Of course that $\sqrt{32}=4+2[0;\overline{1,4}]\in L_f(\Lambda_4)$.
	
	In any case, $$L_f(\Lambda_4)\supset L\cap (-\infty,\sqrt{32}].$$
By induction, suppose that $L_f(\Lambda_{k-1})=L\cap (-\infty,\sqrt{k^2+2k-3}]$, $k\ge 5$. Note that if $\ell\in L\cap (1+\sqrt{k^2+2k-3},\sqrt{k^2+4k})$ then, by the Proposition \ref{prop4} either $\ell\in k+C(\underline{b}^s,k)+C(\underline{b}^s,k), \ s\ge 1$ or $\ell\in k+C(\underline{b}^s,k)+\tilde{C}(\underline{c}^s,k)$. Of course that $\sqrt{k^2+4k}=k+2[0,\overline{1,k}].$ Finally, if $\ell\in L$ is such that 
	$$\sqrt{k^2+2k-3}<\ell \le 1+\sqrt{k^2+2k-3}$$
then $\ell\in k+C(\underline{b}^0,k)+C(\underline{b}^0,k)$ where $\underline{b}^0$ is the empty word. In any case, by Theorem	 \ref{T4.11} we have that $\ell\in L_f(\Lambda_k)$.
\end{proof}

\bibliographystyle{amsplain}

\end{document}